\documentclass[11pt]{amsart}
\usepackage{amsmath,amssymb,amsfonts,amsthm, amscd,indentfirst}
\usepackage{amsmath,latexsym,amssymb,amsmath,
amscd,amsthm,amsxtra}
\usepackage{hyperref}\usepackage{url}
\usepackage{color}
\usepackage{pdfpages}
\usepackage{graphics}
\usepackage{enumitem}

\DeclareMathOperator{\arccosh}{arccosh}
\newtheorem{theorem}{Theorem}[section]
\newtheorem{prop}{Proposition}[section]
\newtheorem{lemma}{Lemma}[section]

\newtheorem{definition}{Definition}[section]

\newtheorem{remark}{\textbf{Remark}}[section]

\def\rr{\mathbb{R}}
\def\ss{\mathbb{S}}
\def\hh{\mathbb{H}}
\def\bb{\mathbb{B}}

\def\O{\Omega}
\def\p{\partial}

\def\a{\alpha}
\def\b{\beta}

\def\p{\partial}

\def\S{{\Sigma}}
\def\<{\langle}
\def\>{\rangle}
\def\div{{\rm div}}
\def\n{\nabla}
\def\G{\Gamma}
\def\ode{\bar{\Delta}}

\def\De{\Delta}

\numberwithin{equation} {section}

\begin{document}

\title[\tiny{A partially overdetermined problem in domains with partial umbilical boundary}]{A partially overdetermined problem in domains with partial umbilical boundary in  space forms}
\author{Jinyu Guo}
\address{School of Mathematical Sciences\\
Xiamen University\\
361005, Xiamen, P.R. China}
\email{guojinyu14@163.com}
\author{Chao Xia}
\address{School of Mathematical Sciences\\
Xiamen University\\
361005, Xiamen, P.R. China}
\email{chaoxia@xmu.edu.cn}
\thanks{This work is supported by NSFC (Grant No. 11871406).
}

\begin{abstract}
In the first part of this paper, we consider a partially overdetermined mixed boundary value problem in space forms and generalize the main result in \cite{GX} into the case of general domains with partial umbilical boundary in space forms. Precisely, we prove that a partially overdetermined problem in a domain with partial umbilical boundary admits a solution if and only if the rest part of the boundary is also part of an umbilical hypersurface. {In the second part of this paper, we prove a Heintze-Karcher-Ros type inequality for embedded hypersurfaces with free boundary lying on a horosphere or an equidistant hypersurface in the hyperbolic space. As an application, we show Alexandrov type theorem for constant mean curvature hypersurfaces with free boundary in these settings.}
\end{abstract}

\date{}
\maketitle


\section{Introduction}
In a celebrated paper \cite{Se}, Serrin initiated the study of the following overdetermined boundary value problem (BVP)
\begin{equation}\label{21}
\begin{cases}{}
\Delta u=1, & \text{in}\ \Omega\\
u=0 ,     & \text{on} \ \partial\Omega\\
\p_\nu u=c, & \text{on}\ \partial\Omega,\\
\end{cases}
\end{equation}
where $\Omega$ is an open, connected, bounded domain in $\mathbb{R}^{n}$ with smooth boundary $\p\O$,   $c\in \rr$ is a constant and $\nu$ is the unit outward normal to $\partial\Omega$.
Serrin proved that if \eqref{21} admits a solution, then  $\Omega$ must be a ball and the solution $u$ is radially symmetric.
Serrin's proof is based on the moving plane method or Alexandrov reflection method, which
has been invented by Alexandrov in order to  prove the famous nowadays so-called Alexandrov's soap bubble theorem \cite{Al}: any closed, embedded hypersurface of constant mean curvature (CMC) must be a round sphere.
\par In space forms\footnote{Throughout this paper, we regard an open hemi-sphere as a spherical space form.}, Serrin's symmetry result was proved in \cite{SJ} and \cite{Mo} by the method of moving plane. A special overdetermined problem in space forms has been considered by Qiu-Xia \cite{QX} by using Weinberger's approach, see also \cite{CV}. We also mention that a corresponding result in the closed sphere case is no longer true, see e.g. \cite{FMW}.

Serrin's overdetermined BVP has close relationship with closed CMC hypersurfaces. Analog to closed CMC hypersurfaces, there are several rigidity results for free boundary CMC hypersurfaces in the Euclidean unit ball $\bb^n$.  Here we use ``free boundary'' to mean a hypersurface which intersects $\ss^{n-1}$ orthogonally. We refer to a recent survey paper \cite{WXsurvey} for details. In particular, Alexandrov type theorem says that a free boundary CMC hypersurface in a half ball must be a free boundary spherical cap. Motivated by this,
 we have proposed  in \cite{GX} the study of a partially overdetemined BVP in a half ball. Precisely, let $\bb^n_+=\{x\in \bb^n: x_n>0\}$ be the half Euclidean unit ball and
$\O\subset \bb^n_+$ be an open bounded, connected domain with boundary $\p\O=\bar \S\cup T$, where $\S\subset\bb^{n}_+$ is a smooth open hypersurface and $T\subset \ss^{n-1}$  meets $\S$ at a common $(n-2)$-dimensional submanifold $\Gamma\subset \ss^{n-1}$.
We have considered the following partially overdetemined BVP in $\O$:
 \begin{equation}\label{Eoverd}
\begin{cases}{}
\Delta u=1 , & \text{in} \ \Omega\subset\bb^n_+,\\
u=0 ,     & \text{on} \ \bar \Sigma,\\
 \p_{\nu}u=c, & \text{on}\ \bar \Sigma,\\
 \p_{\bar N}u=u , & \text{on}\ T,
\end{cases}
\end{equation}
where $\nu$ and $\bar N(x)=x$ are the outward unit normal of $\Sigma$ and $T\subset \ss^{n-1}$ respectively. We have proved the following result.
\begin{theorem}[\cite{GX}]\label{Ethm}
Let $\O$ be as above. Assume  \eqref{Eoverd} admits a weak solution $$u\in W_0^{1,2}(\O, \S)=\{u\in W^{1,2}(\O), u|_{\bar \S}=0\},$$ i.e.,
\begin{eqnarray}\label{weak-form0}
&& \int_\O (\<\n u, \n v\> +v )\, dx -\int_T uv \, dA=0, \hbox{ for all } v\in W_0^{1,2}(\O, \S).
\end{eqnarray}
 Assume further that
$u\in W^{1,\infty}(\O)\cap W^{2,2}(\O).$
 Then   $\O$ must be of the form
 \begin{eqnarray}\label{Esph_cap}
\O_{nc}(a):=\left\{x\in \bb^n_+: |x-a\sqrt{1+(nc)^2}|^2< (nc)^2\right\}, \quad a\in \ss^{n-1} \end{eqnarray}
for some $a\in \ss^{n-1}$ and \begin{eqnarray}\label{Ecap-sol}
u(x)=u_{a,nc}(x):=\frac{1}{2n}(|x-a\sqrt{1+(nc)^2}|^2- (nc)^2).
\end{eqnarray}
\end{theorem}
We remark that $\p\O_{nc}(a)\cap\bb^n_+$ is a free boundary spherical cap. Thus Theorem \ref{Ethm} gives a characterization of  free boundary spherical caps by an overdetermined BVP, which can be regarded as Serrin's analog for the setting of free boundary CMC hypersurfaces in a ball.


\smallskip

In this paper, we will generalize Theorem \ref{Ethm} into the setting of domains with partial umbilical boundary in space forms.

Let $(\mathbb{M}^{n}(K), \bar g)$ be a complete simply-connected Riemann manifold with constant sectional curvature $K$. Up to homoteties we may assume $K=0, 1, -1$; the case $K=0$ corresponds to the case of the Euclidean space $\mathbb{R}^{n}$, $K=1$ is the unit sphere $\mathbb{S}^{n}$ with the round metric and $K=-1$ is the hyperbolic space $\mathbb{H}^{n}$.
We recall some basic facts about umbilical hypersurfaces in $\mathbb{M}^{n}(K)$. It is well-known that  an umbilical hypersuface in space forms has constant principal curvature $\kappa\in \rr$.
By a choice of orientation (or normal vector field $\bar N$), we may assume $\kappa\in [0, \infty)$.  It is also a well-known fact that in $\mathbb{R}^{n}$ and $\ss^n$,  geodesic spheres $(\kappa>0)$ and totally geodesic hyperplanes $(\kappa=0)$ are all complete umbilical hypersurfaces, while in  $\mathbb{H}^{n}$ the family of all complete umbilical hypersurfaces includes geodesic spheres $(\kappa>1)$, totally geodesic hyperplanes $(\kappa=0)$, horospheres $(\kappa=1)$ and equidistant hypersurfaces $(0<\kappa<1)$ (see e.g.\cite{RAF}). We remark that unlike geodesic spheres, the horospheres and the equidistant hypersurfaces are non-compact umbilical hypersurfaces.

We use $S_{K, \kappa}$ to denote an umbilical hypersurface  in $\mathbb{M}^{n}(K)$ with principal curvature $\kappa$. $S_{K, \kappa}$ divides  $\mathbb{M}^{n}(K)$ into two connected components. We  use $B^{{\rm int}}_{K, \kappa}$ to denote the one component whose outward normal is given by the orientation $\bar N$. The other one we denote by $B^{{\rm ext}}_{K, \kappa}$.
Let $\O\subset B^{{\rm int}}_{K, \kappa}$ be a bounded, connected open domain 
   whose boundary $\p\O=\bar{\Sigma}\cup T$, where $\S\subset B^{\rm int}_{K, \kappa}$ is smooth open hypersurface and $T\subset S_{K,\kappa}$ meets $\Sigma$ at a common $(n-2)$-dimensional submanifold $\Gamma$. We refer to Figure 1-3 in Section 2 for the corresponding domains for $K=-1$ (hyperbolic space) and different values $\kappa$. 

Since the Euclidean case $K=0$ has already been handled in \cite{GX}, and the case $\kappa=0$ in $\mathbb{M}^{n}(K)$ has been considered in \cite{Giu} (as a special case of a flat cone),
in this paper we consider the hyperbolic case $K=-1$ with $\kappa>0$ and the spherical case $K=1$ with $\kappa>0$.

We consider the following mixed BVP in $\O\subset B^{\rm int}_{K, \kappa}$:
\begin{equation}\label{bvp}
\begin{cases}{}
\bar{\Delta} u+nKu=1 , & \text{in} \ \Omega,\\
u=0 ,     & \text{on} \ \bar \Sigma,\\
 \p_{\bar N}u=\kappa u , & \text{on}\ T.
\end{cases}
\end{equation}
As we described above, $\bar{N}$ is the unit outward normal of $B^{\rm int}_{K, \kappa}$.


\par If $\kappa>0$, for a general domain, there might not exist a solution to \eqref{bvp}. Also, for a general domain, the maximum principle fails to hold. These are due to the fact that the Robin boundary condition on $T$ has an unfavorable sign.
In our case, 
 we can show that there always exists a unique non-positive solution $u\in C^\infty(\bar \O\setminus \G)\cap C^\a(\bar \O)$ to \eqref{bvp} for some $\a\in (0, 1)$, see Proposition \ref{existence} below.

\begin{remark}
For the other case ${\Omega} \subseteq {B}^{\rm ext}_{K, \kappa}$, $-\bar{N}$ plays the role of the unit outward normal of $B^{\rm ext}_{K, \kappa}$ along $T$. Hence the Robin boundary condition becomes $\p_{(-\bar N)}u=-\kappa u$ on $T$, which has a good sign, i.e. $-\kappa<0$, according to the classical elliptic PDE theory. The existence of weak solution \eqref{bvp} follows directly from the Fredholm alternative theorem (see for example \cite{GT}).
\end{remark}


In this paper, we study the following partially overdetermined BVP in $\O\subset B^{{\rm int}}_{K, \kappa} \quad (B^{{\rm ext}}_{K, \kappa} \hbox{ resp.})$:\begin{equation}\label{overd}
\begin{cases}{}
\bar{\Delta} u+nKu=1 , & \text{in} \ \Omega,\\
u=0 ,     & \text{on} \ \bar \Sigma,\\
 \p_{\nu}u=c, & \text{on}\ \bar \Sigma,\\
 \partial_{\bar{N}}u=\kappa u , & \text{on}\ T.
\end{cases}
\end{equation}
where $\nu$ is the outward unit normal of $\S$.
Our main result is the following
\begin{theorem}\label{mainthm}
Let $\O\subset B^{{\rm int}}_{K, \kappa} \quad (B^{{\rm ext}}_{K, \kappa} \hbox{ resp.})$.
Assume the partially overdetermined BVP \eqref{overd} admits a weak solution $u\in W_0^{1,2}(\O, \S)$, i.e.,
\begin{eqnarray}\label{weak-form0}
&& \int_\O \left(\bar{g}(\bar{\nabla} u, \bar{\nabla} v) +v -nKuv\right)  dx -\kappa\int_T uv \, dA=0, \hbox{ for all } v\in W_0^{1,2}(\O, \S)
\end{eqnarray}together with an additional boundary condition $\p_\nu u=c$ on $\S$.
Assume further that
\begin{eqnarray}\label{reg-ass}
u\in W^{1,\infty}(\O)\cap W^{2,2}(\O).
\end{eqnarray}
\begin{itemize}
  \item [(i)] If $S_{K, \kappa}$ is a horosphere $(K=-1$ and $\kappa=1)$ or an equidistant hypersurface $(K=-1$ and $0<\kappa<1)$ in $\hh^n$, then $\Sigma$ must be part of an umbilical hypersurface with principal curvature $1/(nc)$ which intersects $S_{K, \kappa}$ orthogonally.

\item [(ii)]
  If $S_{K, \kappa}$ is a geodesic sphere in $\hh^n$ or $\ss^n_+$, that is $K=-1$ and $\kappa>1$ or $K=1$ and $\kappa>0$,   then the same conclusion in (i) holds provided $\Omega\subset B^{{\rm int},+}_{K, \kappa} \quad (B^{{\rm ext}, +}_{K, \kappa} \hbox{ resp.})$. 
\end{itemize}

\end{theorem}
Here $B^{{\rm int},+}_{K, \kappa}$ means a half ball, see \eqref{half-ball} below.

\begin{remark}
The above umbilical hypersurface could be a horosphere, an equidistant hypersurface or a geodesic ball. We will give an example in Appendix \ref{appendix} that $\Sigma$ and $T$ are part of two orthogonal horospheres, for which the partially overdetermined BVP \eqref{overd} still admits a solution.

\end{remark}



%

 We remark that we do not assume $\S$ meets $S_{K, \kappa}$ orthogonally a priori. Thus it is impossible to use the Alexandrov reflection method as Ros-Souam \cite{RS}.
On the other hand, since the lack of regularity of $u$ on $\G$, it is difficult to use the maximum principle as Weinberger's \cite{We}.
Higher order regularity up to the interface $\G=\bar \S\cap \bar T$  is a subtle issue for mixed boundary value problems. A regularity result by Lieberman \cite{Liebm} shows that a weak solution $u$ to \eqref{overd} belongs to $C^\infty(\bar \O\setminus \G)\cap C^\a(\bar \O)$ for some $\a\in (0,1)$.
The regularity assumption \eqref{reg-ass} is for technical reasons, that is, we will use an integration method which requires \eqref{reg-ass} to perform integration by parts.

Similar to the Euclidean case \cite{GX}, we use a purely integral method to prove our theorem. The integration makes use of a non-negative weight function $V$, which is given by a multiplier of the divergence of a conformal Killing vector field $X$. 
In the case of geodesic spheres in $\hh^n$ or $\ss^n_+$, we use $X$ defined by \eqref{cfkill1}, which was found in \cite{WX}. In the case of horospheres or equidistant hypersurfaces in $\hh^n$, we use $X$ defined by \eqref{cfkill2}.
The common feature of such conformal Killing vector fields is that it is parallel to the support hypersurfaces.

  By using $X$, we get a Pohozaev-type identity with weight $V$, Proposition \ref{Poho}. Then with the usual $P$-function $P=|\bar{\nabla}u|^2-\frac{2}{n}u+Ku^{2}$, we can show the identity
   {$$\int_\O V u \big|(\bar{\n}^2 u+Kug)-\frac1n (\bar{\Delta} u+nKu) \bar g\big|^2 dx=0.$$} Theorem \ref{mainthm} follows since the $P$-function is subharmonic.

 \medskip

In the second part of this paper, we will use the solution to \eqref{bvp} to study Alexandrov type theorem for embedded free boundary CMC hypersurfaces  in $\hh^n$ supported on  a horosphere or an equidistant hypersurface.

 It is nowadays a routine argument to combine a Minkowski type formula and  a sharp Heintze-Karcher-Ros type inequality to prove Alexandrov type theorem, see e.g. \cite{LX,QX2,Ro, WX}.
In the spirit of Wang-Xia \cite{WX}, we shall first use the solution to \eqref{bvp} to prove the following Heintze-Karcher-Ros type inequality for free boundary hypersurfaces in $\hh^n$ supported on  a horosphere or an equidistant hypersurface. The case of geodesic hyperplane in space forms has been proved by Pyo, see \cite[Theorem 4 and Theorem 10]{PJ}. The case of geodesic spheres in space forms has been shown by Wang-Xia, see \cite[Theorem 5.2 and Theorem 5.4]{WX}.
\begin{theorem}\label{HK}

Let $\hh^n$ be given by the half space model $\{x\in \rr^n_+: x_n>0\}$ with hyperbolic metric $\bar g=\frac{1}{x_n^2}\delta$. Let $\S\subset \hh^n$ be an embedded smooth hypersurface whose boundary $\p \S$ lies on a support hypersurface $S$ (that is, a horosphere or an equidistant hypersurface). Assume $\S$ intersects $S$ orthogonally.
Assume $\S$ has positive normalized mean curvature $H_{1}$ and let $\O$ be the enclosed domain by $\S$ and $S$. Then
\begin{eqnarray}\label{HKR2}
\int_{\S} \frac{1}{x_{n}\cdot H_{1}} dA\ge  \int_\O \frac{n}{x_{n}} dx.
\end{eqnarray}
Moreover, the above equality \eqref{HKR2} holds if and only if $\S$ is part of an umbilical hypersurface which meets $S$ orthogonally.
\end{theorem}
Using the above Heintze-Karcher-Ros type inequality, we are able to reprove  Alexandrov type theorem
for free boundary constant mean curvature or constant higher order mean curvature hypersurfaces in $\hh^n$ supported by horospheres and equidistant hypersurfaces, see Theorem \ref{higherAlex}.

We remark that the Alexandrov type theorem in this setting has been shown by \cite{RAF}, using the classical Alexandrov's reflection method, see also \cite{Wente}.  They  were also able to  handle in \cite{RAF} the general capillary hypersurfaces, that is, constant mean curvature hypersurfaces with contstant contact angle.

\

The rest of the paper is organized as follows. In Section 2, we review the conformal Killing vector fields $X$ we shall use in each case and their properties. In Section 3, we study two kinds of eigenvalue problems in $\O$ in space forms and use them to prove the existence and uniqueness of the mixed BVP \eqref{bvp}. In Section 4, we prove a weighted Pohozaev inequality and then Theorem \ref{mainthm}. In Section 5, We prove Theorem \ref{HK} and the Alexandrov type Theorem \ref{higherAlex}.

\

\section{Conformal Killing vector fields in space forms}

We first introduce the notations. Let us recall that $S_{K,\kappa}$ is an umbilical hypersurface in $\mathbb{M}^{n}(K)$ with principal curvature $\kappa\in [0, \infty)$.
\subsection{Hyperbolic space $\mathbb{H}^{n}$}\
\begin{definition}[\cite{RAF}]
In hyperbolic space $\mathbb{H}^{n}$, we call a {support hypersurface} a complete non-compact umbilical hypersurface, which means geodesic hyperplanes ($\kappa=0$), horospheres ($\kappa=1$) and equidistant hypersurfaces ($0<\kappa<1$).
\end{definition}
A horosphere is a "sphere" whose centre lies at $\partial_{\infty}\mathbb{H}^{n}$. In the upper half-space model \begin{eqnarray}\label{half-space}
\hh^n=\{x=(x_{1}, x_{2},\cdots,x_{n})\in \rr^n_+: x_n>0\},\quad \bar g=\frac{1}{x_n^2}\delta,
\end{eqnarray}
a horosphere, up to a hyperbolic isometry,  is given by the horizontal plane
\begin{equation*}\label{horo12}
L(1)=\{x\in\mathbb{R}^{n}_{+}:x_{n}=1\}
\end{equation*}
By choosing $\bar N=-E_n=(0,\cdots, 0, -1)$, the principal curvature of a horosphere is given by $\kappa=1$. We remark that a horosphere is isometric to a Euclidean plane.

An equidistant hypersurface is a connected component of the set of points equidistant from a given hyperplane. 
 In the half-space model, an equidistant hypersurface, is given by a sloping Euclidean hyperplane $\Pi$ which meets $\partial_{\infty}\mathbb{H}^{n}$ with angle $\theta$ through a point $E_{n}=(0,0,\cdots, 1)\in\mathbb{R}^{n}_{+}$, say
 $$\Pi=\{x\in\mathbb{R}^{n}_{+}: x_1\tan \theta +x_{n}=1\}$$ and by choosing $\bar N$ as the same direction as $(-\tan \theta, 0,\cdots, 0, -1)$, it's principal curvature is $\kappa=\cos\theta\in (0,1)$.

\medskip


Next we clarify the unified notation we will use in each case.
\begin{itemize}
   \item  {\bf Case 1}. If $S_{K,\kappa}$ is a geodesic sphere of radius $R$,  then $\kappa=\coth R\in (1,\infty)$ and let $B_{K, \kappa}^{\rm int}$ denote the geodesic ball enclosed by $S_{K, \kappa}$. By using the poincare ball model \begin{eqnarray}\label{poin-ball}
\mathbb{B}^{n}=\{x\in \rr^n: |x|<1\}, \quad \bar{g}=\frac{4}{(1-|x|^{2})^{2}}\delta,
\end{eqnarray}
we have up to an hyperbolic isometry,
  $$B_{K, \kappa}^{\rm int}=\left\{x\in \bb^n: |x|\le R_\rr:=\sqrt{\frac{1-\arccosh R}{1+\arccosh R}}\right\}.$$
 Moreover, we let
\begin{eqnarray}\label{half-ball}
B^{{\rm int},+}_{K, \kappa}=\left\{x\in B_{K, \kappa}^{\rm int}: x_n>0\right\}
\end{eqnarray}
 be a geodesic half ball.  See Figure 1.

  \item  {\bf Case 2}. If $S_{K,\kappa}$ is a support hypersurface, then $\kappa\in [0, 1]$.
By using the upper half-space model \eqref{half-space}, we have, up to an hyperbolic isometry,
  \begin{equation}
B_{K, \kappa}^{\rm int}=
\begin{cases}{}
\{x\in\mathbb{R}^{n}_{+}: x_n>1\},     & \text{if}\,\, \kappa=1,\\
\{x\in\mathbb{R}^{n}_{+}: x_1\tan \theta +x_{n}>1\}, & \text{if}\,\,\kappa=\cos \theta\in (0, 1).\\
\end{cases}
\end{equation}
See Figure 2 and 3.
\end{itemize}


\begin{figure}
\centering
\includegraphics[height=7cm,width=12cm]{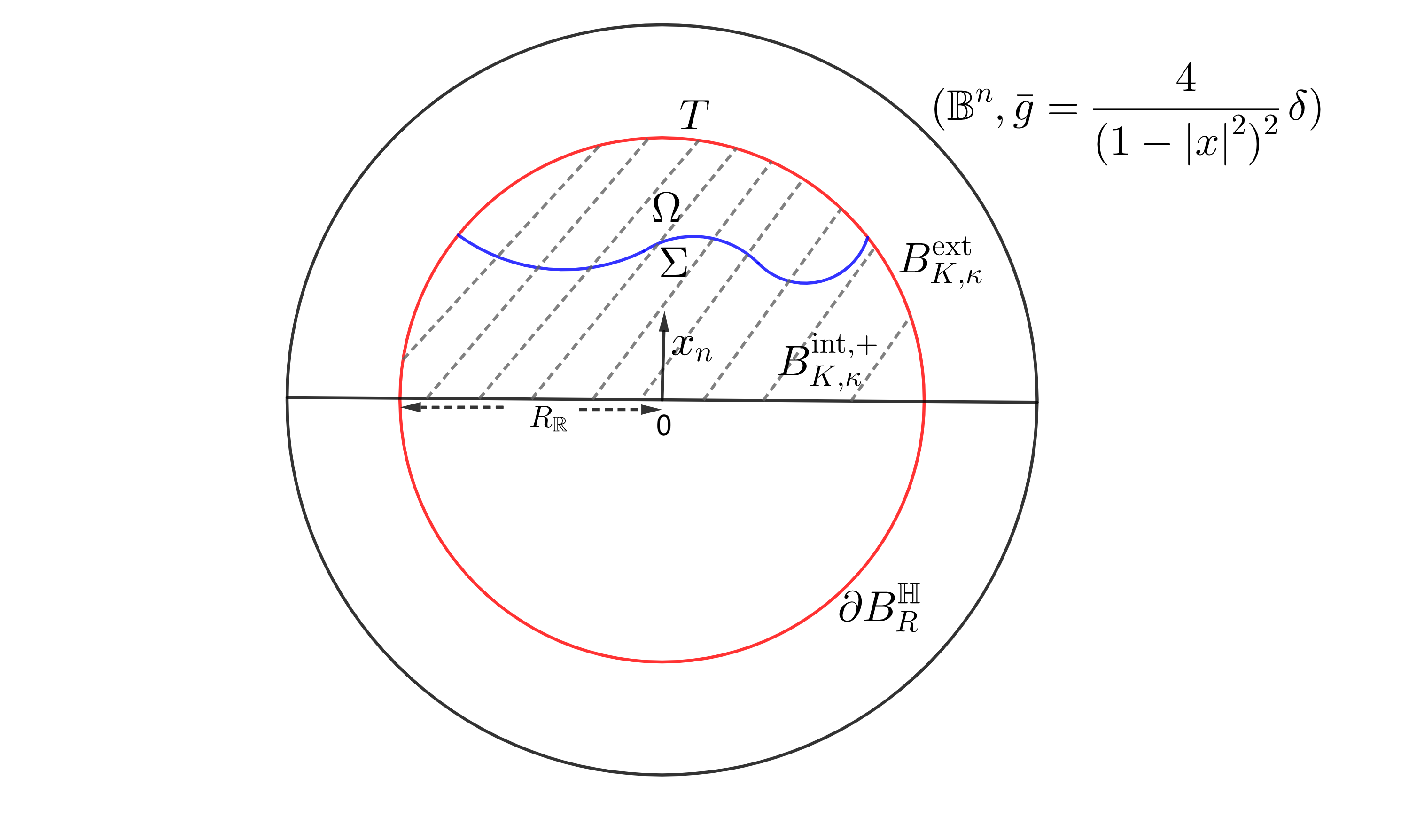}
\caption{ $S_{K, \kappa}$ is a geodesic sphere  with principal curvature $\kappa=\coth R>1$ and the shaded area is $B_{K,\kappa}^{\rm int,+}$}.
\centering
\includegraphics[height=7cm,width=13cm]{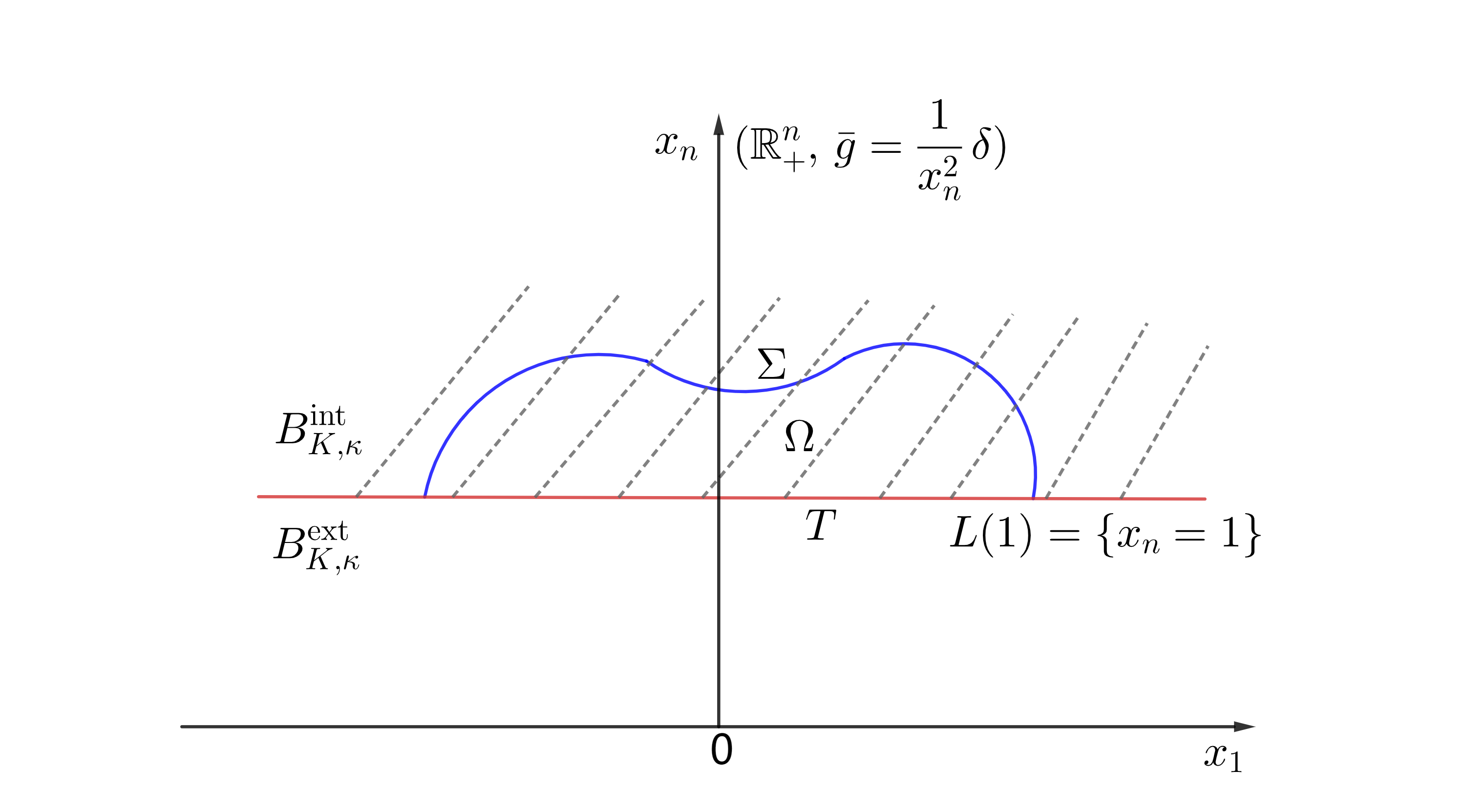}
\caption{$S_{K, \kappa}$ is a horosphere $L(1)$ with principal curvature $\kappa=1$ and the shaded area is $B_{K,\kappa}^{\rm int}$}
\centering
\includegraphics[height=7cm,width=13cm]{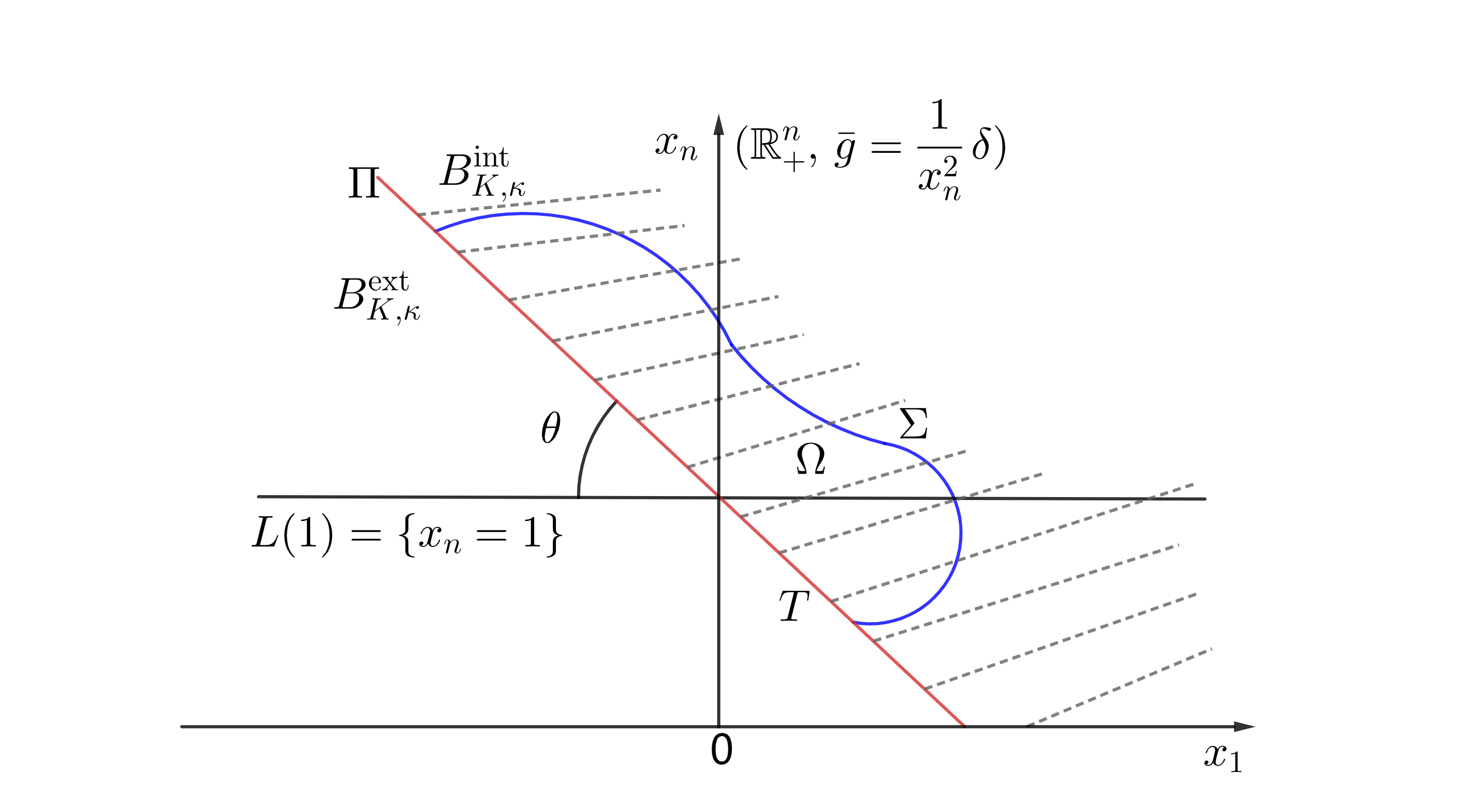}
\caption{$S_{K, \kappa}$ is an equidistant hypersurface $\Pi$ with principal curvature $\kappa=\cos\theta<1$ and the shaded area is $B_{K,\kappa}^{\rm int}$}
\end{figure}


Next we introduce the conformal Killing vector field $X$ in $\mathbb{H}^{n}$ and the weight $V$ we will use later.
\begin{itemize}
  \item   {\bf Case 1}.  $\kappa>1$. In this case, as before we use the Poincar\'e ball model \eqref{poin-ball}. Denote
  \begin{equation}\label{cfkill1}
    X:=\frac{2}{1-R_{\mathbb{R}}^{2}}[x_n x-\frac{1}{2}(|x|^{2}+R_{\mathbb{R}}^{2})E_n], \quad V=\frac{2x_n}{1-|x|^2}.\end{equation}
  \item   {\bf Case 2}. $0<\kappa\leq1$. In this case, as before we use  the upper half-space model \eqref{half-space}. Denote
  \begin{equation}\label{cfkill2}
     X:=x-E_{n}, \quad V=\frac{1}{x_n}.
  \end{equation}
\end{itemize}

\begin{prop}\label{xaa}
$(i)$\,$X$ is a conformal Killing vector field with $L_{X}\bar{g}=V\bar{g}$, namely
\begin{eqnarray}\label{XXaeq1}
\frac{1}{2}(\bar{\nabla}_{i}X_{j}+\bar{\nabla}_{j}X_{i})=V\bar{g}_{ij}.
\end{eqnarray}
\begin{itemize}
  \item [(ii)]$X\mid_{S_{K, \kappa}}$ is a tangential vector field on $S_{K, \kappa}$, i.e.,
\begin{eqnarray}\label{XXaeq2}
\bar{g}(X, \bar{N})=0 \,\,\text{\rm on}\,\, S_{K,\kappa}.
\end{eqnarray}
\end{itemize}
\end{prop}
\begin{proof}\ Case 1. $\kappa>1$, see \cite[Proposition 4.1]{WX}.

Case 2. $0<\kappa\leq1$, we choose an orthonormal basis $\{e_{i}\}_{i=1}^{n}$ in the upper half-space model, 
\begin{equation*}\label{ortho-horo}
  e_{i}=x_{n}E_{i},\quad i=1,\cdots, n.
\end{equation*}
where $\{E_{i}\}_{i=1}^{n}$ is the Euclidean orthonormal basis in $\mathbb{R}^{n}$. Then
\begin{eqnarray}\label{XXaeq1pf}
\frac{1}{2}(\bar{\nabla}_{i}X_{j}+\bar{\nabla}_{j}X_{i})&=&\frac{1}{2}(D_{i}X_{j}+D_{j}X_{i})+X(-\ln x_{n})\bar{g}_{ij}\\
&=&\bar{g}_{ij}+(\frac{1}{x_{n}}-1)\bar{g}_{ij}=\frac{1}{x_{n}}\bar{g}_{ij},\nonumber
\end{eqnarray}
where $D$ is the Levi-Civita connection in $\mathbb{R}^{n}$. We use the relationship of $\bar{\nabla}$ and $D$, that is,
\begin{equation}\label{YZ}
  \bar{\nabla}_{Y}Z=D_{Y}Z+Y(-\ln x_{n})Z+Z(-\ln x_{n})Y-\langle Y,Z\rangle D(-\ln x_{n}).
\end{equation}
Here $\langle\cdot, \cdot\rangle$ is the Euclidean inner product.

On the other hand,
\begin{eqnarray}\label{XXaeq2pf}
\bar{g}(X, \bar{N})=\frac{1}{x_{n}}\langle x-E_{n}, N_{\delta}\rangle=0\quad \text{on}\,\,S_{K, \kappa}
\end{eqnarray}
where $N_{\delta}$ is outward normal to support hypersurface with respect to the Euclidean metric $\delta$.
\end{proof}

\begin{prop}\label{xaa2} \,$V$ satisfies the following properties:
\
\begin{equation}\label{va2}
  \bar{\n}^2 V= -KV  \bar{g },
\end{equation}
\begin{equation}\label{va3}
  \partial_{\bar{N}}V=\kappa V \quad\,\,\text{on}\,\, S_{K,\kappa},
\end{equation}
where $\bar N$ is the outward unit normal of $B_{K,\kappa}^{\rm int}$.
\end{prop}
\begin{proof}\ Case 1. $\kappa>1$, see \cite[Proposition 4.2]{WX}.

Case 2. $0<\kappa\leq1$. 
{ For \eqref{va2}, we take normal coordinates $\{e_{i}\}_{i=1}^{n}$ at $p$ such that $\bar{\nabla}_{e_{i}}e_{j}\mid_{p}=0$. Then
\begin{eqnarray}\label{xnhess2}
  \bar{\nabla}_{e_{i}}\bar{\nabla}_{e_{j}} V&=&\bar{\nabla}_{e_{i}}\bar{\nabla}_{e_{j}}(\frac{1}{x_{n}})=\bar{\nabla}_{e_{i}}\left(\bar{\nabla}_{e_{j}}(\frac{1}{x_{n}})\right)=-e_{i}\left(\bar{g}(E_{n}, e_{j})\right)\nonumber\\
  &=&-\bar{g}(\bar{\nabla}_{e_{i}}E_{n}, e_{j})=E_{n}(\ln x_{n})\bar{g}_{ij}=\frac{1}{x_{n}}\bar{g}_{ij},
\end{eqnarray}
where we use formula \eqref{YZ}.}

For \eqref{va3}, we compute
\begin{eqnarray}\label{xnhess3}
 \partial_{\bar{N}}V= \partial_{\bar{N}}(\frac{1}{x_{n}})=-\frac{1}{x_n^{2}}\langle E_{n},\bar{N}\rangle=-\frac{1}{x_{n}^{2}}\langle E_{n},x_{n}N_{\delta}\rangle=\frac{1}{x_{n}}\cos\theta
\end{eqnarray}
Here we use the fact that $\langle E_{n}, N_{\delta}\rangle=-\cos\theta$ on the support hypersurface $S_{K, \kappa}$.
 \end{proof}


\subsection{Spherical space $\mathbb{S}^{n}$}\

In this subsection, we sketch  the necessary  modifications in the case that the ambient space is the spherical space $\ss^{n}$.
We use the model $$\left(\rr^{n},\, \bar g_{\ss}=\frac{4}{(1+|x|^2)^2}\delta\right) $$ to represent $\ss^{n}\setminus\{\mathcal{S}\}$, the unit sphere without the south pole.
{
Therefore, if $S_{K,\kappa}$ is a geodesic sphere of radius $R$, then in the above model
$$S_{K,\kappa}=\left\{x\in \rr^n: |x|= R_{\rr}:=\sqrt{\frac{1-\cos R}{1+\cos R}}\right\}.$$}
Then $\kappa=\cot R>0$, for $R<\frac{\pi}{2}$.
Let $B^{{\rm int}}_{K, \kappa}$ be a geodesic ball enclosed by $S_{K,\kappa}$ and $B^{{\rm int},+}_{K, \kappa}$ be the geodesic half ball given by
\begin{equation}\label{half-ball-sphere}
  B^{{\rm int},+}_{K, \kappa}=\{x\in B^{{\rm int}}_{K, \kappa}:\, x_n>0 \}.
\end{equation}

Let $X$ be the vector field
 \begin{equation}\label{Sph-killing}
   X=\frac{2}{1+R_{\mathbb{R}}^{2}}[x_n x-\frac{1}{2}(|x|^{2}+R_{\mathbb{R}}^{2})E_n], \quad V=\frac{2x_n}{1+|x|^2}.
 \end{equation}
It has been shown in \cite{WX} that $X$ and $V$ also satisfy Propositions \ref{xaa} and \ref{xaa2}.


\

\section{Mixed BVP in space forms}
  From this section on, let $\O$ be a bounded, connected open domain in $B^{{\rm int}}_{K, \kappa}$ whose boundary $\p\O$ consists two parts $\bar{\Sigma}$ and $T=\p\O\setminus \bar{\Sigma}$, where $T\subset S_{K,\kappa}$ is smooth and meets $\Sigma$ at a common $(n-2)$-dimensional submanifold $\Gamma$. {If $S_{K,\kappa}$ is a geodesic sphere, then we assume further that $\O\subset B^{{\rm int},+}_{K, \kappa}$. For notation simplicity and unification, in the following sections, we use $\O\subset B^{{\rm int}}_{K, \kappa}$ to indicate that $\O\subset B^{{\rm int},+}_{K, \kappa}$ in the case that $S_{K,\kappa}$ is a geodesic sphere.}

 We consider  the following two kinds of eigenvalue problems in $\O$.\\
\noindent{\bf I. Mixed Robin-Dirichlet eigenvalue problem}
 \begin{equation}\label{mixRD_eigen}
\begin{cases}{}
\bar{\Delta} u=-\lambda u, &
\hbox{ in } \Omega,\\
u= 0,&\hbox{ on }\bar \S,\\
 \p_{\bar N}u=  \kappa u, &\hbox{ on } T.
\end{cases}
\end{equation}
The first Robin-Dirichlet eigenvalue can be variational characterized by
\begin{eqnarray}\label{lamet1}
\lambda_1=\inf_{0\ne u\in W_0^{1,2}(\O, \S)}\frac{\int_\O \bar{g}(\bar{\nabla} u, \bar{\nabla} u) dx-\kappa\int_T u^2 dA}{\int_\O u^2 dx}.
\end{eqnarray}

\noindent{\bf II. Mixed Steklov-Dirichlet eigenvalue problem.} (see e.g. \cite{Agr, BKPS})
 \begin{equation}\label{mixSD_eigen}
\begin{cases}{}
\bar{\Delta} u+nKu=0, &
\hbox{ in } \Omega,\\
u= 0,&\hbox{ on } \bar \S,\\
 \p_{\bar N}u= \mu \kappa u, &\hbox{ on } T.
\end{cases}
\end{equation}
The mixed Steklov-Dirichlet eigenvalues can be considered as the eigenvalues of the Dirichlet-to-Neumann map
\begin{eqnarray*}
\mathcal{L}: &L^2(T)\to &L^2(T)\\
&u\mapsto & \frac{1}{\kappa}\p_{\bar N}\hat u
\end{eqnarray*}
where $\hat u\in W_0^{1,2}(\O, \S)$ is the extension of $u$ to $\O$ satisfying $\bar{\Delta} \hat{u}+nK\hat{u}=0$ in $\Omega$ and $\hat{u}=0$ on $\S$. According to the spectral theory for compact, symmetric linear  operators, $\mathcal{L}$ has a discrete spectrum $\{\mu_i\}_{i=1}^\infty$ (see e.g. \cite{Agr} or \cite{ BKPS}),
$$0<\mu_1\le \mu_2 \le\cdots \to +\infty.$$

The first eigenvalue $\mu_1$ can be variational characterized by
\begin{eqnarray*}
\mu_1=\inf_{0\ne u\in W_0^{1,2}(\O, \S)}\frac{\int_\O \bar{g}(\bar{\nabla} u, \bar{\nabla} u) dx-nK\int_{\Omega}u^{2}dx}{\kappa\int_T u^2 dA}.
\end{eqnarray*}

In our case, we have
\begin{prop}\label{mixSD_eigen1}\ If $S_{K,\kappa}$ is geodesic sphere in $\mathbb{H}^{n}$ or $\mathbb{S}^{n}$ and $\Omega\subseteq B_{K,\kappa}^{\rm int,+}$, then
\begin{itemize}\item[(i)] $\lambda_1(\O)\ge nK$ and $\lambda_1=nK$ if and only if $\O=B_{K,\kappa}^{\rm int,+}$.
\item[(ii)] $\mu_1(\O)\ge 1$ and $\mu_1=1$ if and only if $\O=B_{K,\kappa}^{\rm int,+}$.
\end{itemize}
\end{prop}

\begin{proof}
We proceed exactly as \cite{GX}. If $\O=B_{K,\kappa}^{\rm int,+}$, one checks that $u=V\ge 0$ indeed solves \eqref{mixRD_eigen} with $\lambda=nK$ and \eqref{mixSD_eigen} with $\mu=1$. Since $u=V$ is a non-negative solution, it must be the first eigenfunction and hence $\lambda_1(B_{K,\kappa}^{\rm int,+})=nK$ and $\mu_1(B_{K,\kappa}^{\rm int,+})=1.$

On the other hand, for $\O \subset B_{K,\kappa}^{\rm int,+}$, by the variational characterization and a standard argument of doing zero extension, one sees $\lambda_1(\O)\ge \lambda_1(B_{K,\kappa}^{\rm int,+})=nK$ and $\mu_1(\O)\ge\mu_1(B_{K,\kappa}^{\rm int,+})=1.$

{ If $\O\subsetneqq B_{K,\kappa}^{\rm int,+}$, then the Aronszajn unique continuity theorem \cite{Aron} implies $\lambda_1(\O)>\lambda_1(B_{K,\kappa}^{\rm int,+})=nK$. In fact, we extend the first Robin-Dirichlet eigenfunction $u$ in $\O$ to $\tilde{u}$ in $B_{K,\kappa}^{\rm int,+}$ by defining $\tilde{u}=0$ outside $\bar \O$. Then $\tilde{u}$ is the first Robin-Dirichlet eigenfunction in $B_{K,\kappa}^{\rm int,+}$ by its variational characterization \eqref{lamet1}. However, the Aronszajn unique continuity theorem\footnote{Let $\mathcal L$ be a second order elliptic operator with $C^{3}$ coefficients. If $\mathcal{L}u=0$ in an open connected domain $\O$ and $u=0$ in an open subset of $\O$, then $u=0$ is identically zero in $\O$.} would imply that $u=0$ is identically zero on $B_{K,\kappa}^{\rm int,+}$. This is a contradiction that $u$ is the first eigenfunction in $\O$.}

For $\mu_1$, it has been proved in  \cite[Proposition 3.1.1]{BKPS}, that $\mu_1(\O)>\mu_1(B_{K,\kappa}^{\rm int,+})=1$.
\end{proof}

\begin{prop}\label{eigen-horo}\ If $S_{K,\kappa}$ is a horosphere or an equidistant hypersurface in $\mathbb{H}^{n}$ and $\Omega\subseteq B_{K,\kappa}^{\rm int}$, then
\begin{equation}\label{horesphere}
 \lambda_{1}(\Omega)>-n \,\, \text{\rm and}\,\,\mu_1(\O)>1.
\end{equation}
\end{prop}
\begin{proof}
We first take an orthonormal basis $\{e_{i}\}_{i=1}^{n}$ in the upper half space model 
\begin{equation*}\label{ortho-horo}
  e_{i}=x_{n}E_{i},\quad i=1,\cdots, n.
\end{equation*}
\noindent By using the divergence theorem, we get
\begin{eqnarray}\label{horo-div2}
\int_{\Omega}\div_{\bar{g}}(u^{2}e_{n})dx=\int_{\partial\Omega}u^{2}\bar{g}(e_{n},\nu)dA=\int_{T}u^{2}\bar{g}(e_{n},\bar{N})dA=-\cos\theta\int_{T}u^{2}dA
\end{eqnarray}
where we also use $u=0$ on $\Sigma$ and the fact $\langle\bar{N}, E_{n}\rangle=-x_{n}\cos\theta$ on $T$ and $\theta\in[0,\frac{\pi}{2})$.

\noindent On the other hand,
\begin{equation}\label{horo-div}
  \int_{\Omega}\div_{\bar{g}}(u^{2}e_{n})dx=\int_{\Omega}e_{n}(u^{2})dx+\int_{\Omega}u^{2}\div_{\bar{g}}(e_{n})dx=\int_{\Omega}e_{n}(u^{2})dx-(n-1)\int_{\Omega}u^{2}dx.
\end{equation}
Combining \eqref{horo-div2} with \eqref{horo-div}, we have
\begin{eqnarray}\label{horo-div3}
\cos\theta\int_{T}u^{2}dA&=&\int_{\Omega}(n-1)u^{2}-e_{n}(u^{2})dx=\int_{\Omega}(n-1)u^{2}-2ue_{n}(u)dx\\
&\leq&\int_{\Omega}(n-1)u^{2}+2|u||\bar{\nabla} u|dx\nonumber\\
&\leq&\int_{\Omega}(n-1)u^{2}+|u|^{2}+\bar{g}(\bar{\nabla} u, \bar{\nabla} u) dx\nonumber\\
&=&\int_{\Omega}nu^{2}+\bar{g}(\bar{\nabla} u, \bar{\nabla} u) dx\nonumber
\end{eqnarray}
Since $0\ne u\in W_0^{1,2}(\O, \S)$, we know that the above equality is strict, namely,
\begin{eqnarray}\label{horo-div4}
\cos\theta\int_{T}u^{2}dA<\int_{\Omega}nu^{2}+\bar{g}(\bar{\nabla} u, \bar{\nabla} u)dx
\end{eqnarray}
Recall that $\kappa=\cos\theta$. Therefore, we complete this proof by taking infimum for $u$.

\end{proof}

Using Proposition \ref{mixSD_eigen1} (ii) and \ref{eigen-horo}, we show the existence and uniqueness of mixed BVP \eqref{bvp}.
\begin{prop}\label{existence}Let $f\in C^\infty(\O)$, $q\in C^\infty(T)$ and $\O\subsetneqq B_{K,\kappa}^{\rm int}$. Then the mixed BVP
\begin{equation}\label{mixSD_nonh}
\begin{cases}{}
\bar{\Delta }u+nKu=f, &
\hbox{ in } \Omega,\\
u= 0,&\hbox{ on } \bar\S,\\
 \p_{\bar N}u= \kappa u+q, &\hbox{ on } T.
\end{cases}
\end{equation}
admits a unique weak solution $u\in W^{1,2}_0(\O, \S)$. Moreover, $u\in C^\infty(\bar \O\setminus \G)\cap C^\a(\bar \O)$ for some $\a\in (0,1)$.
\end{prop}

\begin{proof} The weak solution to \eqref{mixSD_nonh} is defined to be $u\in W_0^{1,2}(\O, \S)$ such that
\begin{equation}\label{Buv}
B[u,v]:= \int_\O \bar{g}(\bar{\nabla} u, \bar{\nabla} v)  dx -\kappa\int_T uv \, dA-nK\int_{\Omega}uvdx \hbox{ for all } v\in W_0^{1,2}(\O, \S),
\end{equation}
\begin{eqnarray}\label{weak-form}
&&B[u,v]=\int_\O -fv\, dx+\int_T qv \, dA\,\hbox{ for all } v\in W_0^{1,2}(\O, \S).
\end{eqnarray}
From Proposition \ref{mixSD_eigen1}(ii) and Proposition \ref{eigen-horo}, we know $1-\frac{1}{\mu_1}>0$. There holds
\begin{equation}\label{dss}
 \kappa\int_{T}u^{2}dA\leq\frac{1}{\mu_{1}}(\int_{\Omega}\bar{g}(\bar{\nabla} u, \bar{\nabla} u)dx-nK\int_{\Omega}u^{2}dA),
\end{equation}
By using \eqref{Buv}, \eqref{dss} and Proposition \ref{mixSD_eigen1} (i) and \ref{eigen-horo}, there exists a positive constant $\beta$ such that
\begin{eqnarray}\label{coercive}
B[u, u]\ge (1-\frac{1}{\mu_1}) (\int_\O\bar{g}(\bar{\nabla} u, \bar{\nabla} u)\, dx-nK\int_{\Omega}u^{2}dx)\ge  \beta\|u\|^{2}_{W_{0}^{1,2}(\O, \S)}
\end{eqnarray}
\par Thus $B[u,v]$ is coercive on $W_0^{1,2}(\O, \S)$.
The standard Lax-Milgram theorem holds for the weak formulation to \eqref{mixSD_nonh}. Therefore,  \eqref{mixSD_nonh} admits a unique weak solution $u\in W_0^{1,2}(\O, \S)$.

The regularity $u\in  C^\infty(\bar \O\setminus \G)$ follows from the classical regularity theory for elliptic equations and $u\in C^\a(\bar \O)$ has been proved by Lieberman \cite[Theorem 2]{Liebm}.
Note that the global wedge condition in \cite[Theorem 2]{Liebm} is satisfied for the domain $\O$ whose boundary parts $\S$ and $T$ meet at a common in smooth $(n-2)$-dimensional manifold, see page 426 of  \cite{Liebm}.
\end{proof}
\begin{prop}\label{mp}Let $u$ be the unique solution to \eqref{mixSD_nonh} with $f\ge 0$ and $q\le 0$.
Then either $u\equiv0$ in $\O$ or $u<0$ in $\O\cup T$.
\end{prop}
\begin{proof} Since the Robin boundary condition has an unfavorable sign, we cannot use the maximum principle directly.
Since $u_{+}=\text{max}\{u , 0\}\in W_0^{1,2}(\O, \S)$, we can use it as a test function in the weak formulation \eqref{weak-form} to get
\begin{alignat*}{2}
\int_{\Omega}-f u_+ \, dx + \int_T q u_+ \, dA&=\int_{\Omega}\bar{g}(\bar{\nabla} u_{+}, \bar{\nabla} u_{+}) dx-\kappa\int_T (u_+)^2\, dA-nK\int_{\Omega}(u_{+})^{2}dx.
\end{alignat*}
Since $f\ge 0$ and $q\le 0$, we have
\begin{alignat*}{2}
\int_{\Omega}-f u_+ \, dx + \int_T q u_+ \, dA\le 0.
\end{alignat*}
On the other hand, it follows from Proposition \ref{mixSD_eigen1} (i) and \eqref{horesphere} that
\begin{alignat*}{2}
\int_{\Omega}\bar{g}(\bar{\nabla} u_{+}, \bar{\nabla} u_{+}) dx-\kappa\int_T (u_+)^2\, dA-nK\int_{\Omega}(u_{+})^{2}dx \ge (\lambda_1-nK)\int_\O (u_+)^2 dx\ge 0.
\end{alignat*}
From above, we conclude that $u_+\equiv 0$, which means $u\le 0$ in $\O$. Finally, by the strong maximum principle, we get either $u\equiv0$ in $\Omega$ or $u<0$ in $\Omega\cup T$.
\end{proof}

\begin{prop} Let $e^{T}$ be a tangent vector field to $T$. Let $u$ be the unique solution to (\ref{bvp}). Then
\begin{equation}\label{bdry-prop}
 \bar{\n}^{2}u(\bar N, e^T)=0 \quad on\ T.
\end{equation}
\end{prop}
\begin{proof}
By differentiating the equation $\p_{\bar N}u=\kappa u$ with respect to $e^{T}$, we get
\begin{alignat*}{2}\label{18}
\kappa\bar{\n}_{e^{T}}(u)&={e^{T}}(\bar{g}(\bar{\n} u,\bar N))= \bar{\n}^{2}u(\bar N, e^T)+\bar{g}(\bar{\n} u, \bar{\nabla}_{e^{T}}\bar N)\\
&=\bar{\n}^{2}u(\bar N, e^T)+h^{S_{K, \kappa}}(\bar{\nabla} u,e^{T})=\bar{\n}^{2}u(\bar N, e^T)+\kappa\bar{g}(\bar{\n} u, e^T).
\end{alignat*}
Here we use the fact that $S_{K, \kappa}$ is an umbilical hypersurface with principal curvature $\kappa$. The assertion \eqref{bdry-prop} follows.
\end{proof}

\section{Partially overdetermined BVP in space forms}
In this section we will use a method totally based on integral identities and inequalities to prove Theorem \ref{mainthm}. The proof follows closely our previous paper \cite{GX}. The main ingredient is Propositions \ref{xaa} and \ref{xaa2}.

First we introduce $P$ function as follows,
\begin{equation}\label{60}
P:=\bar{g}(\bar{\nabla} u, \bar{\nabla} u)-\frac{2}{n}u+Ku^{2}.
\end{equation}

\begin{prop}\label{sub-harm}$\bar{\Delta} P=2\big|(\bar{\n}^2 u+Kug)-\frac1n (\bar{\Delta} u+nKu) \bar g\big|^2\ge 0 \hbox{ in }\O.$
\end{prop}
\begin{proof} By direct computation and using $\bar{\Delta} u+nKu=1$,
 \begin{alignat*}{2}
\bar{\Delta} P(x)&= 2|\bar{\n}^2 u|^{2}+2\bar{g}(\bar{\n} u, \bar{\n} \bar{\De} u)+2\overline{Ric}(\bar{\nabla} u,\bar{\nabla} u)-\frac{2}{n}\bar{\Delta} u+2K(\bar{g}(\bar{\nabla} u, \bar{\nabla} u)+u\bar{\Delta} u)
\\&=2\big|(\bar{\n}^2 u+Kug)-\frac1n (\bar{\Delta} u+nKu) \bar g\big|^2
\\&\ge 0.
\end{alignat*}
\end{proof}
Due to the lack of regularity, we need the following formula of integration by parts, see \cite[Lemma 2.1]{PA}. (The original statement \cite[Lemma 2.1]{PA} is for a sector-like domain in a cone. Nevertheless, the proof is applicable in our case). We remark that a general version of integration-by-parts formula for Lipschitz domains has been stated in some classical book by Grisvard \cite[Theorem 1.5.3.1]{Grisv}. However, it seems not enough for our purpose.

\begin{prop}[\cite{PA}, lemma 2.1]\label{int-by-parts}
Let $F:\O \to \rr^n$ be a vector field such that
\begin{eqnarray*}F\in C^1(\O\cup \S\cup T)\cap L^2(\O)\quad \hbox{ and } \quad \div(F)\in L^1(\O).
\end{eqnarray*}
Then \begin{eqnarray*}
\int_\O \div(F) dx=\int_{\S} \bar{g}(F, \nu) dA+ \int_{T} \bar{g}(F, \bar N) dA.
\end{eqnarray*}
\end{prop}

We first prove a Pohozaev-type identity for \eqref{overd}.

\begin{prop}\label{Poho}
Let  $u$ be the unique solution to \eqref{overd}. Then we have
\begin{equation}\label{Peq}
\int_{\Omega}V(P-c^{2})dx=0.
\end{equation}
\end{prop}

\begin{proof}
{ First of all, we remark that by regularity in Proposition \ref{existence}, $u\in C^\infty(\O\cup \S\cup T)$, that is, $u$ is smooth away from the corner $\G$. }Moreover, due to our assumption $u\in W^{1,\infty}(\O)\cap W^{2,2}(\O)$, Proposition \ref{int-by-parts} can be applied in all the following integration by parts.

Now we consider the following differential identity
\begin{eqnarray}\label{Peq1}
&&\div(uX-\bar{g}(X,\bar{\nabla} u)\bar{\nabla} u)\\
&&=\bar{g}(X, \bar{\nabla} u)+u\div X-\bar{\nabla} X(\bar{\nabla} u,\bar{\nabla} u)
-\frac{1}{2}\bar{g}(X,\bar{\nabla}\bar{g}(\bar{\nabla} u, \bar{\nabla} u))-\bar{g}( X, \bar{\nabla} u)\bar{\Delta }u\nonumber\\
&&=nVu-V\bar{g}(\bar{\nabla} u, \bar{\nabla} u)-\frac{1}{2}\bar{g}( X,\bar{\nabla}\bar{g}(\bar{\nabla} u, \bar{\nabla} u))+nK\bar{g}( X,\bar{\nabla}(\frac{1}{2}u^{2})).\nonumber
\end{eqnarray}
where we use equation $\bar{\De} u+nKu=1$ and \eqref{XXaeq1}.

Integrating by parts and using \eqref{XXaeq2} and boundary conditions \eqref{overd}, we see that
\begin{eqnarray}\label{Peq2}
&&-c^{2}\int_{\Sigma}\bar{g}( X,\nu) dA-\int_{T}\bar{g}( X,\bar{\nabla} u) u_{\bar{N}}dA\\
&=&\int_{\Omega}\left(nVu-V\bar{g}(\bar{\nabla} u, \bar{\nabla} u)+\frac{1}{2}\bar{g}(\bar{\nabla} u, \bar{\nabla} u)\div X-\frac{nK}{2}u^{2}\div X\right)dx-\frac{ c^{2}}{2}\int_{\Sigma}\bar{g}( X,\nu) dA\nonumber.
\end{eqnarray}

 It follows that
 \begin{eqnarray}\label{Peq3}
&&\int_{\Omega}\left(nVu+(\frac{n}{2}-1)V\bar{g}(\bar{\nabla} u, \bar{\nabla} u)-\frac{n^{2}K}{2}Vu^{2}\right)dx\\
&&=-\frac{c^{2}}{2}\int_{\Sigma}\bar{g}( X,\nu) dA-\kappa\int_{T}\bar{g}( X,\bar{\nabla} u) udA.\nonumber
\end{eqnarray}

Using \eqref{XXaeq1} and \eqref{XXaeq2} yields
\begin{eqnarray}
&&\int_\S \frac12c^2\bar{g}( X, \nu) dA= \frac12c^2\left(\int_\O {\rm div}Xdx -\int_T\bar{g}(X, \bar N)dA\right)= \frac n 2 c^2\int_\O Vdx,\label{Peq5}\\
&&\kappa\int_T \ \bar{g}( X, \bar{\n}u ) udA=\kappa\int_T \bar{g}( X^{T},\bar{\nabla}(\frac{1}{2}u^{2})) dA\label{Peq6}\\
&&\qquad\qquad\qquad\qquad\quad=\frac{\kappa}{2}\int_{\Gamma}u^{2}\bar{g}( X^{T},\mu) ds-\frac{\kappa}{2}\int_{T}u^{2}\div_{T}X^{T}dA\nonumber\\
&&\qquad \qquad\qquad\qquad\quad=\frac{\kappa(1-n)}{2}\int_{T}Vu^{2}dA.\nonumber
\end{eqnarray}
In the last equality we also used $u=0$ on $\G$ and $\div_{T}X^{T}=(n-1)V$.

 To achieve \eqref{Peq}, we do a further integration by parts and apply \eqref{va2} and \eqref{va3} get
 \begin{eqnarray*}
 \int_\O  V \bar{g}(\bar{\nabla} u, \bar{\nabla} u)dx&=&\int_{T} V u(u_{\bar{N}})dA-\int_\O\left(\bar{g}(\bar{\nabla} V,\bar{\nabla}(\frac{1}{2}u^{2})) +Vu\bar{\Delta} u\right)dx\\
 &=&\kappa\int_{T}Vu^{2}dA-\frac{1}{2}\int_{T}(V)_{\bar{N}} u^{2}dA+\int_{\Omega}(\frac{1}{2}u^{2}\bar{\Delta} V-Vu\bar{\Delta} u)dx
 \\&=&\kappa\int_{T}Vu^{2}dA-\frac{\kappa}{2}\int_{T}Vu^{2}dA+\int_{\Omega}\frac{1}{2}u^{2}(-nKV)-Vu(1-nKu)dx
 \\&=&\frac{\kappa}{2}\int_{T}Vu^{2}dA+\frac{nK}{2}\int_{\Omega}Vu^{2}dx-\int_{\Omega}Vudx.
\end{eqnarray*}
It follows that
 \begin{eqnarray}\label{Peq10}
\frac{\kappa}{2}\int_{T}Vu^{2}dA=\int_\O \left( V\bar{g}(\bar{\nabla} u, \bar{\nabla} u)-\frac{nK}{2}Vu^{2}+Vu\right)dx.
\end{eqnarray}
Substituting \eqref{Peq5}-\eqref{Peq10} into \eqref{Peq3}, we arrive at \eqref{Peq}.
\end{proof}

\begin{prop}\label{int-id}
Let $u$ be the unique solution to \eqref{overd} such that $u\in W^{1,\infty}(\O)\cap W^{2,2}(\O)$. Then
{\begin{eqnarray}\label{Xeq0}
\int_\O V u \big|(\bar{\n}^2 u+Kug)-\frac1n (\bar{\Delta} u+nKu) \bar g\big|^2\, dx=0.
\end{eqnarray}
}
\end{prop}
\begin{proof}
{ Since $u\in W^{1,\infty}(\O)\cap W^{2,2}(\O)$, then \begin{eqnarray}\label{eq-new1}
\Delta P= 2\big|(\bar{\n}^2 u+Kug)-\frac1n (\bar{\Delta} u+nKu) \bar g\big|^2 \in L^1(\O).
\end{eqnarray}
 It follows that $$\div(V u \bar{\n} P- P\bar{\n} (V u))\in L^1(\O) \quad \hbox{ and }\quad (V u \bar{\n} P- P\bar{\n} (V u))\in L^2(\O).$$
 }
Firstly, we consider the following differential identity
\begin{eqnarray}
&&\div(Vu\bar{\nabla }P-P\bar{\nabla}(Vu))+c^{2}\div(V\bar{\nabla} u-u\bar{\nabla} V)\nonumber\\
\qquad&=&Vu\bar{\Delta }P-P\bar{\Delta}(Vu)+c^{2}(V\bar{\Delta} u-u\bar{\Delta} V)\nonumber\\
\qquad&=&Vu\bar{\Delta} P-2P\bar{g}(\bar{\nabla} V,\bar{\nabla} u)-(P+c^{2})u\bar{\Delta} V+(c^{2}-P)V\bar{\Delta} u\nonumber\\
\qquad&=&Vu\bar{\Delta} P-2P\bar{g}(\bar{\nabla} V,\bar{\nabla} u)-(P+c^{2})u(-nKV)+(c^{2}-P)V(1-nKu)\nonumber\\
\qquad&=&Vu\bar{\Delta} P-2P\bar{g}(\bar{\nabla} V,\bar{\nabla} u)+2nKPuV-(P-c^{2})V
.\label{Xeq1}
\end{eqnarray}
where we use the fact that $\bar{\nabla}^{2}V=-KV\bar{g}$ and the equation $\bar{\Delta} u+nKu=1$ in $\Omega$.

Applying divergence theorem in \eqref{Xeq1} and  boundary conditions \eqref{overd}, we have
\begin{eqnarray}\label{Xeq291}
&&\int_{\Sigma}(-P+c^{2})Vu_{\nu}dA+\int_{T}\left(VuP_{\bar{N}}-P(Vu)_{\bar{N}}\right)dA\\
&=&\int_{\Omega}\left(Vu\bar{\Delta} P-2P\bar{g}(\bar{\nabla} V,\bar{\nabla} u)+2nKPuV-(P-c^{2})V\right)dx\nonumber\\
&=&\int_{\Omega}\left(Vu\bar{\Delta} P-2P\bar{g}(\bar{\nabla} V,\bar{\nabla} u)+2nKPuV\right)dx.\nonumber
\end{eqnarray}
where the last equation we use Pohozaev-type identity \eqref{Peq}.

Noting that $P=c^2$ on $\S$. It follows from \eqref{Xeq291}
\begin{eqnarray}\label{Xeq3}
\quad\int_{\Omega}Vu\bar{\Delta }Pdx=\int_{\Omega}\left( 2P\bar{g}(\bar{\nabla} V, \bar{\nabla} u)-2nKPuV\right) dx+\int_{T}(VuP_{\bar{N}}-P(Vu)_{\bar{N}})dA.
\end{eqnarray}

Now we compute the first term of \eqref{Xeq3}, we have
\begin{alignat}{2}\label{Xeq1.6}
\int_\O2\bar{g}(\bar{\nabla} V,\bar{\nabla} u) Pdx&=2\int_\O \bar{g}(\bar{\nabla} V,\bar{\nabla} u)\left(\bar{g}(\bar{\nabla} u, \bar{\nabla} u)-\frac{2}{n}u+Ku^{2}\right)dx\nonumber\\
&=2\int_\O \bar{g}(\bar{\nabla} V,\bar{\nabla} u)\bar{g}(\bar{\nabla} u, \bar{\nabla} u)dx-\frac{2}{n}\int_{\Omega}\bar{g}(\bar{\nabla} V,\bar{\nabla} u^{2}) dx+2K\int_{\Omega}\bar{g}(\bar{\nabla} V,\bar{\nabla} u) u^{2}dx\nonumber\\
&=2\int_{T}(\partial_{\bar{N}}V)u\bar{g}(\bar{\nabla} u, \bar{\nabla} u)dA-2\int_\O\left(\bar{\Delta} V\bar{g}(\bar{\nabla} u, \bar{\nabla} u)+2\bar{\nabla}^{2}u(\bar{\nabla} V,\bar{\nabla} u)\right)udx\nonumber\\
&\quad\quad -\frac{2}{n}\int_{T}(\partial_{\bar{N}}V)u^{2}dA+\frac{2}{n}\int_{\Omega}\bar{\Delta} Vu^{2}dx+2K\int_{\Omega}\bar{g}(\bar{\nabla} V,\bar{\nabla} u) u^{2}dx\nonumber\\
&=2\kappa\int_{T}Vu\bar{g}(\bar{\nabla} u, \bar{\nabla} u)dA+2nK\int_\O V\bar{g}(\bar{\nabla} u, \bar{\nabla} u)udx-2\int_\O\bar{\nabla}^{2}u(\bar{\nabla} V,\bar{\nabla}(u^{2}))dx\nonumber\\
&\quad\quad -\frac{2}{n}\kappa\int_{T}Vu^{2}dA+\frac{2}{n}\int_{\Omega}(-nKV)u^{2}dx+2K\int_{\Omega}\bar{g}(\bar{\nabla} V,\bar{\nabla} u) u^{2}dx\nonumber\\
&=2\kappa\int_{T}(\bar{g}(\bar{\nabla} u, \bar{\nabla} u)-\frac{u}{n})uVdA+2K\int_{\Omega}(n\bar{g}(\bar{\nabla} u, \bar{\nabla} u)-u)uVdx\nonumber\\
&\quad\quad -2\int_{\Omega}\left(\bar{\nabla}^{2}u(\bar{\nabla} V,\bar{\nabla}(u^{2}))-K\bar{g}(\bar{\nabla} V,\bar{\nabla} u) u^{2}\right)dx,
\end{alignat}
where we have used the fact $\bar{\nabla}^{2}V=-KV\bar{g}$ and $\partial_{\bar{N}}V=\kappa V$ on $T$.

Now, we compute the last term of (\ref{Xeq1.6}) by using $\bar{\nabla}^{2}V=-KV\bar{g}$, Ricci identity and \eqref{bdry-prop}
\begin{alignat}{2}\label{Xeq1.7}
&-2\int_{\Omega}\bar{\nabla}^{2}u(\bar{\nabla} V,\bar{\nabla}(u^{2}))-K\bar{g}(\bar{\nabla} V,\bar{\nabla} u) u^{2}dx\\
&=-2\int_{T}u^{2}\bar{\nabla}^{2}u(\bar{\nabla} V,\bar{N})dA
 +2\int_{\Omega}\left(\bar{g}(\bar{\nabla}^{2}V,\bar{\nabla}^{2}u) +\bar{g}(\bar{\nabla} V,\bar{\nabla}\bar{\Delta} u)
+\overline{Ric}(\bar{\nabla} V,\bar{\nabla} u)\right)u^{2}dx\nonumber\\
&\qquad +2K\int_{\Omega}\bar{g}(\bar{\nabla} V,\bar{\nabla} u) u^{2}dx\nonumber\\
&=-2\int_{T}u^{2}(V_{\bar{N}})\bar{\nabla}^{2}u(\bar{N},\bar{N})dA+2\int_{\Omega}\left(-KV\bar{\Delta} u+\bar{g}(\bar{\nabla} V,\bar{\nabla}(1-nKu))+(n-1)K\bar{g}(\bar{\nabla} V,\bar{\nabla} u)\right)u^{2}dx\nonumber\\
&\qquad +2K\int_{\Omega}\bar{g}(\bar{\nabla} V,\bar{\nabla} u) u^{2}dx\nonumber\\
&=-2\kappa\int_{T}u^{2}V\bar{\nabla}^{2}u(\bar{N},\bar{N})dA-2K\int_{\Omega}V(\bar{\Delta} u)u^{2}dx\nonumber\\
&=-2\kappa\int_{T}u^{2}V\bar{\nabla}^{2}u(\bar{N},\bar{N})dA-2K\int_{\Omega}V(1-nKu)u^{2}dx.\nonumber
\end{alignat}

Next, we compute the boundary term of (\ref{Xeq3})
\begin{alignat}{2}\label{Xeq1.9}
\int_{T}VuP_{\bar{N}}-(Vu)_{\bar{N}}PdA&=\int_{T}(P_{\bar{N}}-2\kappa P)uVdA=\int_{T}\left(2\bar{\nabla}^{2}u(\bar{\nabla} u,\bar{N})+2\kappa(\frac{u}{n}-\bar{g}(\bar{\nabla} u, \bar{\nabla} u))\right)uVdA\nonumber\\
&=\int_{T}\left(2\kappa u\bar{\nabla}^{2}u(\bar{N},\bar{N})+2\kappa(\frac{u}{n}-\bar{g}(\bar{\nabla} u, \bar{\nabla} u))\right)uVdA.
\end{alignat}
In the last equality we used (\ref{bdry-prop}) and also $\partial_{\bar{N}}u=\kappa u$ on $T$.

Substituting (\ref{Xeq1.6})-(\ref{Xeq1.9}) into (\ref{Xeq3}) and noticing \eqref{eq-new1}, we get the conclusion (\ref{Xeq0}).

\end{proof}

\noindent{\bf Proof of Theorem \ref{mainthm}.}We note that  in both cases, $V>0$.  In the case that $S_{K, \kappa}$ is a horosphere or an equidistant hypersurface,   $V=\frac{1}{x_{n}}>0$. In the case that $S_{K, \kappa}$ is a geodesic sphere in $\mathbb{H}^{n}$ or $\mathbb{S}^{n}_{+}$,  $V>0$ in $\Omega$ since $\Omega\subseteq B^{{\rm int},+}_{K, \kappa}$, see \eqref{half-ball} and \eqref{half-ball-sphere}.

{From Propositions \ref{mp} and \ref{sub-harm} as well as $V>0$ in $\Omega$, we have
\begin{eqnarray}\label{equiv}
V u\big|(\bar{\n}^2 u+Kug)-\frac1n (\bar{\Delta} u+nKu) \bar g\big|^2\le 0 \hbox{ in }\O.
\end{eqnarray}
It follows from Proposition \ref{int-id}
that $$V u \big|(\bar{\n}^2 u+Kug)-\frac1n (\bar{\Delta} u+nKu) \bar g\big|^2\equiv 0 \hbox{ in }\O.$$}
Since $u<0$ in $\O$ by Proposition \ref{mp}, we see immediately that
 $\bar{\n}^{2}u$ is proportional to the metric $g$
in $\Omega$. Since $\bar{\Delta} u+nKu=1$, we get $$\bar{\nabla}^{2}u=(\frac{1}{n}-Ku)\bar{g}.$$
From this, by restricting on $\Sigma$ and using $u=0$ on $\Sigma$, we get $h_{ij}=\frac{1}{nc}g_{ij}$
which means $\S$ must be part of an umbilical hypersurface with principal curvature $\frac{1}{nc}$.
\qed

\

\section{Heintze-Karcher-Ros inequality and Alexandrov Theorem}
In this section, We shall first use the solution to \eqref{bvp} prove the Heintze-Karcher-Ros type inequality for free boundary hypersurfaces in $\hh^n$ supported on  a horosphere or an equidistant hypersurface.

\noindent{\bf Proof of Theorem \ref{HK}.}\

The proof follows closely \cite[Theorem 5.2]{WX}.
Denote $\Omega$ be a bounded connected domain enclosed by $\Sigma$ and $S_{K,\kappa}$ whose boundary $\partial\Omega=\Sigma\cup T$.
Let $u$ be a solution of the following mixed BVP,
\begin{equation}\label{bvp11}
\begin{cases}{}
\bar{\Delta} u-nu=1 , & \text{in} \ \Omega,\\
u=0 ,     & \text{on} \ \bar \Sigma,\\
 \p_{\bar N}u=\kappa u , & \text{on}\ T.
\end{cases}
\end{equation}
where $\bar N$ is the unit outward normal of $B_{K,\kappa}^{\rm int}$. The existence and regularity of $u$ has been proved in Proposition \ref{existence}.

Using \eqref{va2}, we have
 \begin{eqnarray}\label{fact11}
&&\bar{\Delta} (\frac{1}{x_{n}})-\frac{n}{x_{n}}=0 ,\quad \bar \Delta (\frac{1}{x_{n}})\bar g-\bar \n^2 (\frac{1}{x_{n}})+\frac{1}{x_{n}}\overline{{\rm Ric}}=0.
\end{eqnarray}
Combining \eqref{bvp11} and \eqref{fact11}, we apply Green's formula
\begin{eqnarray}\label{green}
\int_\O \frac{1}{x_{n}} dx&=& \int_\O \frac{1}{x_{n}}(\bar{\Delta} u-nu)- (\ode(\frac{1}{x_{n}})-\frac{n}{x_{n}}) u dx\nonumber
\\&=&\int_{\p\O} \frac{1}{x_{n}} u_\nu- \partial_\nu(\frac{1}{x_{n}}) u dA\nonumber
\\&=&\int_{\S}\frac{1}{x_{n}} u_{\nu} dA+\int_{T} \frac{1}{x_{n}} u_{\bar N}-u\partial_{\bar N}(\frac{1}{x_{n}}) dA\nonumber
\\&=&\int_{\S} \frac{1}{x_{n}} u_{\nu} dA.
\end{eqnarray}
where we use the fact \eqref{va3}.

Using H\"older's inequality for the RHS of \eqref{green}, we have
\begin{eqnarray}\label{first1}
\left(\int_\O \frac{1}{x_{n}}dx\right)^2 \leq \int_{\S} \frac{1}{x_{n}}H_{1}u_{\nu}^2 dA\int_\S \frac{1}{x_{n}\cdot H_{1}} dA.
\end{eqnarray}

Applying the weighted Reilly type formula in \cite{LX, QX}, (see also \cite[Theorem 5.1]{WX}) in our case with $V=\frac{1}{x_{n}}$, we see
\begin{eqnarray}\label{qx22}
\frac{n-1}{n}\int_\O \frac{1}{x_{n}} dx&=&  \int_\O \frac{1}{x_{n}} (\ode u-nu)^2 dx-\frac{1}{n} \int_\O \frac{1}{x_{n}} (\ode u-nu)^2dx\nonumber\\
&\geq &\int_\O \frac{1}{x_{n}} \left((\ode u-nu)^2-|\bar \n^2 u-u\bar g|^2\right)dx\nonumber
 \\&=&\int_{\S} \frac{n-1}{x_{n}} H_{1}u_{\nu}^2 dA +\int_T\left(h^{S_{K,\kappa}}-\kappa\cdot \bar{g}\right)\left(\n u-x_{n}{\n (\frac{1}{x_{n}})}u,\n u-x_{n}{\n (\frac{1}{x_{n}})}u\right)dA\nonumber
  \\&=&\int_{\S} \frac{n-1}{x_{n}} H_{1}u_{\nu}^2 dA
\end{eqnarray}
where we use \eqref{va3} and $S_{K,\kappa}$ is an umbilical hypersurface with principal curvature $\kappa$.

Combining \eqref{first1} and \eqref{qx22}, we get \eqref{HKR2}.

If the equality in \eqref{qx22} holds, we get $\bar{\nabla}^{2}u=(\frac{1}{n}+u)\bar{g}$ in $\Omega$. Since $u=0$ on $\Sigma$, we know that $\Sigma$ must be part of an umbilical hypersurface.
\qed

Denote $h$ and $H_{r}$ to be second fundamental form and normalized $r$-th mean curvature of $\Sigma$ respectively. Precisely,
$h(X, Y)=\bar{g}(\bar{\nabla}_{X}\nu, Y)$ and $H_{r}:=\dbinom{n-1}{r}^{-1}S_{r}$, where $S_r$ is given by
\begin{equation}
  S_{r}=\sum_{i_{1}<i_2<\cdots<i_{r}}\kappa_{i_{1}}\kappa_{i_{2}}\cdots\kappa_{i_{r}}\quad\text{for all}\,\, r=1,\cdots n-1.
\end{equation}
where $\kappa_{1}, \kappa_{2}, \cdots, \kappa_{n-1}$ are principal curvature of $\Sigma$ in $\mathbb{H}^{n}$. As convention, we define $H_{0}=1$.

Let $T_{r}(h)=\frac{\p S_{r+1}}{\p h}$ be the Newton transformation. We state the following properties of $T_{r}$.
\begin{lemma}[\cite{Reilly, BarCar}]\label{lemma-0}
 For each $0 \leq r \leq n-2$
\begin{itemize}
  \item [(1)] The Newton tensor $T_{r}$ is divergence-free, i.e., $\div T_{r}=0$;
  \item [(2)]{trace}$\left(T_{r}\right)=(n-1-r)S_{r}$;
  \item [(3)]{trace}$\left( T_{r}h\right)=(r+1)S_{r+1}$;
  \item [(4)]{trace}$\left(T_{r}h^{2} \right)=S_{1}S_{r+1}-(r+2)S_{r+2}$.
\end{itemize}
\end{lemma}
Next we prove the Minkowski formulas for free boundary hypersurfaces in $\mathbb{H}^{n}$ supported on a support hypersurface.
\begin{prop}\label{high-Alx55}
\begin{equation}\label{high-Alx}
 \int_{\Sigma}\frac{H_{k-1}}{x_{n}}dA=\int_{\Sigma}H_{k}\bar{g}(X,\nu)dA,\,\,\text{for all}\,\,k=1,\cdots,n-1.
\end{equation}
\end{prop}

\begin{proof}
Let $X^T$ be the tangential projection of $X$ on ${\S}$. We know that $X^T\perp \bar{N}$  along $\p \S$ by \eqref{XXaeq2}.
Let $\{e_\a\}_{\a=1}^{n-1}$ be an orthonormal frame on ${\S}$. From Proposition \ref{xaa} (i), we have that
 \begin{eqnarray}\label{confKillhoro}
\frac12\left[\bar{\n}_\a (X^T)_\b+\bar{\n}_\b (X^T)_\a\right]=\frac{1}{x_{n}} \bar{g}_{\a\b}-h_{\a\b}\bar g(X,\nu).
\end{eqnarray}
Multiplying $T_{k-1}^{\a\b}(h)$ to \eqref{confKillhoro} and integrating by parts on $\S$, from Lemma \ref{lemma-0}, we get
\begin{eqnarray*}
&&\int_\S (n-k) \frac{1}{x_{n}}S_{k-1}-kS_k\bar g(X,\nu) dA
\\&=&\int_\S T_{k-1}^{\a\b}\bar{\n}_\a (X^T)_\b dA=\int_\S \bar{\nabla}_{\alpha}(T_{k-1}^{\a\b}X^T)_\b dA
\\&= &\int_{\p\S} T_{k-1}(X^T, \bar{N}) ds=0.
\end{eqnarray*}
In the last equality, we use the fact that $S_{K,\kappa}$ is an umbilical hypersurface, $\bar{N}$ is a principal direction of $h$, it is also a  principal direction of the Newton tensor $T_{k-1}$ of $h$, which implies that $T_{k-1}(X^T, \bar{N})=0$. The above Proposition is completed by the definition of $H_{k}$.
\end{proof}

Now we use the above Minkowski formulas \eqref{high-Alx} to prove Alexandrov type Theorem for free boundary CMC hypersurfaces  in $\mathbb{H}^{n}$ supported by a support hypersurface.
\begin{theorem}\label{Alex-thm}
Assume $S_{K,\kappa}$ is a horosphere or an equidistant hypersurface. Let $x: \Sigma\to \mathbb{H}^{n}$ be an embedded smooth CMC hypersurface into $B_{K,\kappa}^{int}$ (or $B_{K,\kappa}^{ext}$) whose boundary $\partial\Sigma$ lies on $S_{K,\kappa}$. Assume $\Sigma$ meets $S_{K,\kappa}$ orthogonally. Then $\S$ must be part of an umbilical hypersurface.
\end{theorem}
\begin{proof}
Consider the upper half-space model. In the case that $T$ is a horosphere, let $D_R=\{x\in\mathbb{R}^{n}_{+}: |x|<R\}$.  In the case that $T$ is an equidistant, let $D_R=\{x\in\mathbb{R}^{n}_{+}: |x-b|<R\}$ where $b=(1,0,\cdots, 0)$.

Since $\Sigma$ is a compact hypersurface, we take $R$  large enough (small resp. ) such that $\Sigma \subseteq D_R$ (or $D_R\cap\Sigma=\emptyset$ resp.) when $\Sigma$ lies in $B_{K,\kappa}^{int}$ ( $B_{K,\kappa}^{ext}$ resp.).
Let $D_R$ shrink (expand resp.) along radial direction in the Euclidean sense, until it touches $\S$ at some point $p$ at a first time.
By our choice of $D_R$, it does not intersect with $T$ orthogonally.
Since $D_R$ does not intersect with $T$ orthogonally, but $\S$ does, we see that $p$ is an interior point of $\S$.
It follows that $H_{1}=H_{1}(p)\ge 0$.
If $H_{1}=0$, then the maximum principle
implies that $\S$ must be some totally geodesic, which is a contradiction since $\Sigma$ is perpendicular to $S_{K,\kappa}$ by hypothesis. Therefore, $H_{1}$ is positive.

Let $\Omega$ be a bounded connected domain enclosed by $\Sigma$ and $S_{K, \kappa}$ whose boundary $\partial\Omega=\Sigma\cup T$.
Using Proposition \ref{xaa} (i) and (ii), we see
\begin{eqnarray}
&&\int_{\Omega}\div Xdx=\int_\O \frac{n}{x_{n}} dx\label{Alx1}\\
&&\int_{\Omega}\div Xdx=\int_\Sigma \bar g(X, \nu)dA+\int_T \bar g(X, \bar{N})dA=\int_\Sigma\bar g(X, \nu)dA\label{Alx2}\\
&&=\frac{1}{H_{1}}\int_\Sigma H_{1}\bar g(X, \nu)dA=\frac{1}{H_{1}}\int_\S \frac{1}{x_{n}} dA=\int_\S\frac{1}{x_{n}\cdot H_{1}}dA.\label{Alx3}
\end{eqnarray}
The above equation \eqref{Alx3} we use Proposition \ref{high-Alx55} with $k=1$. Then the conclusion is from \eqref{Alx1}, \eqref{Alx2}and Theorem \ref{HK}.
\end{proof}

\
Next we use the method of Ros \cite{Ro} and Koh-Lee \cite{KL} to prove higher order Alexandrov Theorem for embedded free boundary CMC hypersurfaces  in $\hh^n$ supported by a support hypersurface.
\begin{theorem}\label{higherAlex}
Assume $S_{K,\kappa}$ is a horosphere or an equidistant hypersurface. Let $x: \Sigma\to \mathbb{H}^{n}$ be an isometric proper immersion smooth hypersurface into $B_{K,\kappa}^{int}$ (or $B_{K,\kappa}^{ext}$) whose boundary $\partial\Sigma$ lies on $S_{K,\kappa}$. Assume $\Sigma$ meets $S_{K,\kappa}$ orthogonally.
\begin{itemize}
\item[(i)] If $x$ is an embedding and has nonzero constant higher order mean curvatures $H_k$, $1\le k\le n-1$. Then $\S$  is part of an umbilical hypersurface.
\item[(ii)] If $x$  has nonzero constant curvature quotient,  i.e.,
$$\frac{H_k}{H_l}=constant, \quad H_l>0, \quad 1\le l<k\le n-1.$$ Then $\S$ is part of an umbilical hypersurface.
\end{itemize}
\end{theorem}

\begin{proof}
Since $H_{k}$ is a constant, arguing as the beginning of the proof of Theorem \ref{Alex-thm}, we get $H_{k}>0$ by the compactness of $\Sigma$. The principal curvature are continuous functions. Therefore, we can choose a connected component such that $H_{1},\cdots,H_{k-1}$ are all positive at any point.
According to  \cite{Ro} and the Newton-MacLaurin inequality, we have for each $1\leq r\leq k$,
  \begin{equation}\label{inequ}
0\leq H_{r}^{\frac{1}{r}}\leq H_{r-1}^{\frac{1}{r-1}}\leq\cdots\leq H_{1}
  \end{equation}
\begin{equation}\label{inequ33}
0\leq\frac{H_{r}}{H_{r-1}}\leq\frac{H_{r-1}}{H_{r-2}}\leq\cdots\leq
\frac{H_{1}}{H_{0}}=H_{1}
\end{equation}
with the equality holding only at umbilical points on $\Sigma$. Here $H_0=1$ by convention.
\noindent It follows from \eqref{inequ} that
\begin{equation}\label{inequ2}
  \frac{1}{H_{1}}\leq H_{k}^{-\frac{1}{k}}
\end{equation}
Using Theorem \ref{HK} and \eqref{inequ2}, we have
\begin{equation}\label{HKeq}
 \int_{\Omega}\frac{n}{x_{n}}dx\leq\int_{\Sigma}\frac{1}{x_{n}\cdot H_{1}}dA\leq\int_{\Sigma} H_{k}^{-\frac{1}{k}}\cdot\frac{1}{x_{n}}dA
\end{equation}

On the other hand, by Proposition \ref{xaa}(i) and (ii)
\begin{eqnarray}
&&H_{k}\int_{\Omega}\div Xdx=H_{k}\int_{\Omega}\frac{n}{x_{n}}dx\label{Alx11}\\
&&H_{k}\int_{\Omega}\div Xdx=H_{k}\int_{\Sigma}\bar{g}(X,\nu)dA=\int_{\Sigma}H_{k}\bar{g}(X,\nu)dA=\int_{\Sigma}\frac{ H_{k-1}}{x_{n}}dA\geq\int_{\Sigma} \frac{ H_{k}^{\frac{k-1}{k}}}{x_{n}}dA\label{Alx12}
\end{eqnarray}
where in the last equality and inequality we have used \eqref{high-Alx} and \eqref{inequ} respectively.

Combining \eqref{HKeq}-\eqref{Alx12}, we get \eqref{inequ} equality holds on $\Sigma$. Therefore, $\Sigma$ is part of an umbilical hypersurface. The proof of (i) is finished.

Arguing as the beginning of the proof of Theorem \ref{Alex-thm}, there is a point $p$ on $\Sigma$ such that all the principal curvatures are positive. Therefore, $H_{k}$ and $H_{l}$ are positive at $p$. Since $\alpha=\frac{H_{k}}{H_{l}}$ is constant and $H_{l}$ is positive on $\Sigma$, then $H_{k}>0$ on $\Sigma$ and $\alpha>0$.

By Newton-MacLaurin inequality, we have
\begin{equation}\label{Alx-high}
  0<\alpha=\frac{H_{k}}{H_{l}}\leq \frac{H_{k-1}}{H_{l-1}}
  \end{equation}
Using Proposition \ref{high-Alx55} and $H_{k}=\alpha H_{l}$, we have
\begin{eqnarray}
&&\int_{\Sigma}\frac{H_{k-1}}{x_{n}}dA=\int_{\Sigma}H_{k}\bar{g}(X,\nu)dA=\int_{\Sigma}\alpha H_{l}\,\bar{g}(X,\nu)dA\label{high-Alex14}\\
&&=\alpha\int_{\Sigma} H_{l}\,\bar{g}(X,\nu)dA=\alpha\int_{\Sigma}\frac{H_{l-1}}{x_{n}}dA\nonumber
\end{eqnarray}
where in the last equality we have used \eqref{high-Alx} again.
Combining \eqref{Alx-high} with \eqref{high-Alex14}, we get
\begin{equation}\label{Alx-high15}
\alpha=\frac{H_{k}}{H_{l}}= \frac{H_{k-1}}{H_{l-1}}=const\quad \text{on}\,\,\Sigma
\end{equation}
Proceeding inductively, and taking $p=k-l$, we see that
\begin{equation}\label{Sp}
\frac{H_{p+1}}{H_{1}}=\frac{H_{p}}{H_{0}}=H_{p}
\end{equation}
By using \eqref{inequ33}, we have
\begin{equation}\label{inequ44}
\frac{H_{p+1}}{H_{p}}=\frac{H_{p}}{H_{p-1}}=\cdots=
\frac{H_{1}}{H_{0}}=H_1
\end{equation}
Therefore, $\Sigma$ is part of an umbilical hypersurface. The proof of (ii) is completed.

\end{proof}


\appendix

\section{An example for BVP with horospheres as boundary}\label{appendix}
In the section we give an example that the partial overdetermined problem \eqref{overd} admits a solution if the domain is bounded by two orthogonal horospheres.

We use the Poincar\'e ball model \eqref{poin-ball}. Let $L_{1}$ and $L_{2}$ be two horospheres as follows
\begin{eqnarray}
&& L_{1}:=\left\{x\in \bb^n\mid |x'|^{2}+(x_{n}-\frac{1}{2})^{2}=\frac{1}{4}\right\}\label{L1}\\
&& L_{2}:=\left\{x\in\bb^n\mid|x'|^{2}+(x_{n}+\frac{1}{3})^{2}=\frac{4}{9}\right\}\label{L2}
\end{eqnarray}
$L_1$ and $L_2$ are mutually orthogonal. See Figure 4.

\begin{figure}
\centering
\includegraphics[height=7cm,width=13cm]{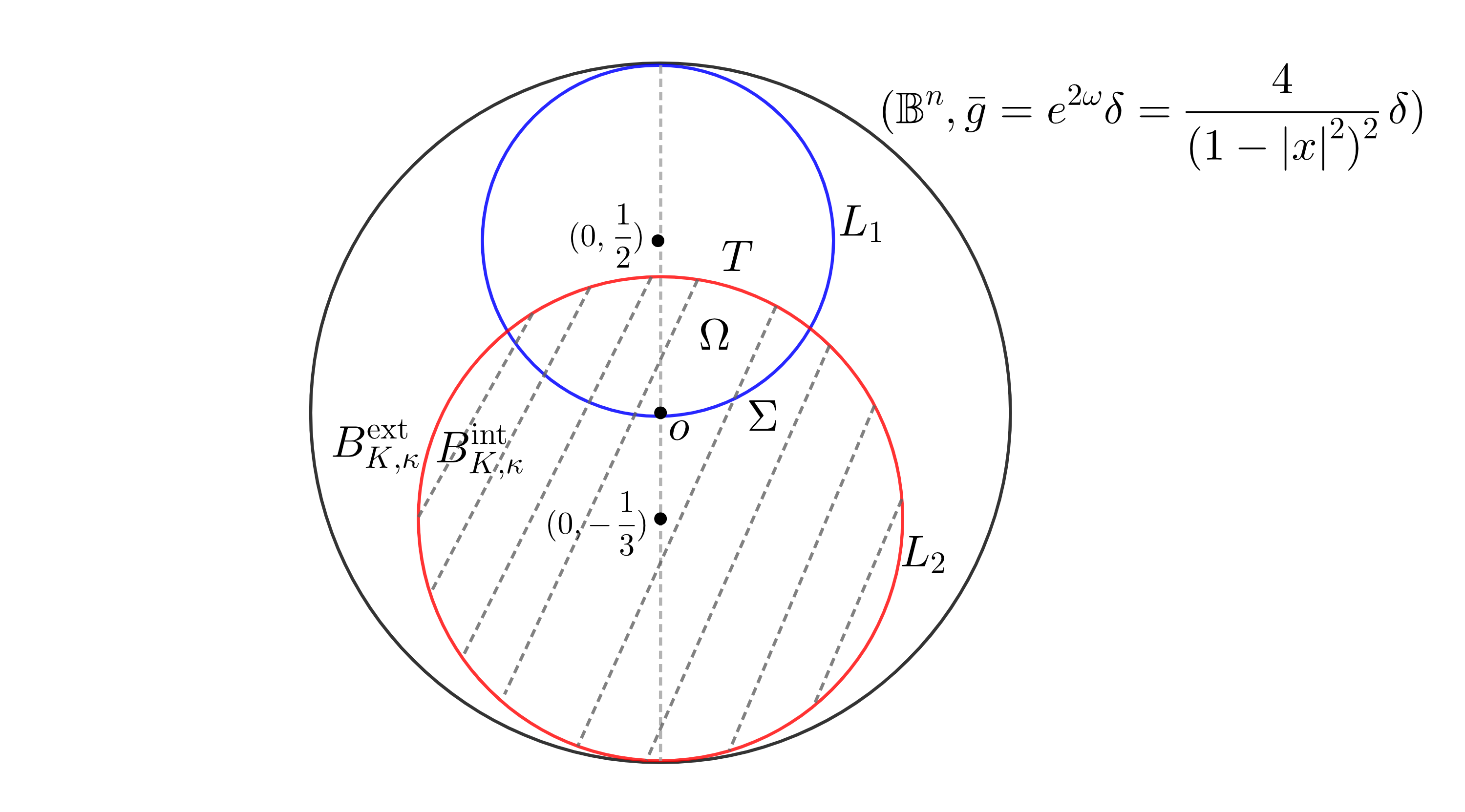}
\caption{$L_{1}$ and $L_{2}$ are two horospheres  with principal curvature $\kappa=1$ and the shaded area is $B_{K,\kappa}^{\rm int}$.}
\end{figure}

The domain $\Omega$ is bounded by $\bar{\Sigma}$ and $T$, that is $\partial\Omega=\bar{\Sigma}\cup T$, where $\Sigma\subseteq L_{1}$ and $T\subseteq L_{2}$.

Let
\begin{eqnarray*}
&& V_{0}:=\frac{1+|x|^{2}}{1-|x|^{2}},\quad V=\frac{2x_{n}}{1-|x|^{2}}.\end{eqnarray*}
It is direct to see
\begin{eqnarray}\label{V0}
 |x|^{2}=\frac{V_{0}-1}{V_{0}+1}\label{V0}, \quad x_{n}=\frac{V}{V_{0}+1}.
\end{eqnarray}
One has the following
\begin{prop}[\cite{WX}, Proposition 4.2]\label{V0Vn}
\begin{eqnarray}
&& \bar{\nabla}^{2}V_{0}=V_{0}\bar{g}, \quad \bar{\nabla}^{2}V=V\bar{g}.
\end{eqnarray}
\end{prop}
\begin{prop}[\cite{WX} Proposition 4.3]\label{field}
For any tangential vector field $Z$ on $\hh^{n}$,
\begin{eqnarray}
&&\bar \n_{Z} V_0= \bar g(x, Z);\label{dV0}
\\&&\bar \n_{Z} V=e^{-\omega}\bar g(Z,E_{n})+ e^{-2\omega}\bar g(x, E_{n})\bar g(x, Z).\label{dV1}
\end{eqnarray}
Here $e^{2\omega}=\frac{4}{(1-|x|^2)^2}$.
\end{prop}
Let
\begin{equation*}
  u:=\frac{1}{n}(V_{0}-V-1)
\end{equation*}
By Proposition \ref{V0Vn}, we have
\begin{equation*}
  \bar{\Delta}u-nu=1\quad \text{in} \,\,\Omega.
\end{equation*}
Combining \eqref{V0} with \eqref{L1}, we get
\begin{equation*}
u=0\quad\text{on}\,\, \Sigma\subset L_1.
\end{equation*}
Since $\Sigma\subseteq L_{1}$ and $\bar{g}=e^{2\omega}\delta$, then the unit outward normal vector of $\Sigma$ is $$\nu=2e^{-\omega}(x', x_{n}-\frac{1}{2}).$$

Using \eqref{dV0} and \eqref{dV1}, one checks that on $\Sigma$,
\begin{eqnarray}
&&\bar{\nabla}_{\nu}u=\frac{1}{n}(\bar{\nabla}_{\nu}V_{0}-\bar{\nabla}_{\nu}V)=\frac{1}{n}\bar{g}(x,\nu)(1-e^{-2\omega}\bar{g}(x,E_{n}))-\frac{1}{n}e^{-\omega}\bar{g}(\nu,E_{n})\\
&&=\frac{2}{n}e^{\omega}(|x|^{2}-\frac{1}{2}x_{n})(1-x_{n})-\frac{2}{n}(x_{n}-\frac{1}{2})\nonumber\\
&&=\frac{1}{n}\nonumber.
\end{eqnarray}
In the last equality above we have used \eqref{L1}.

On the other hand, since $T\subseteq L_{2}$ and $\bar{g}=e^{2\omega}\delta$, then the unit outward normal vector of $T$ is $$\bar{N}=\frac{3}{2}e^{-\omega}(x',x_{n}+\frac{1}{3}).$$

Using \eqref{L2}, we get
\begin{equation}\label{f}
u=\frac{1}{n}(\frac{1+|x|^{2}}{1-|x|^{2}}-\frac{2x_{n}}{1-|x|^{2}}-1)=\frac{5|x|^{2}-1}{n(1-|x|^{2})} \quad \text{on}\,\,T.
\end{equation}
Using \eqref{dV0}, \eqref{dV1} and \eqref{L2}, one checks that on $T$,
\begin{eqnarray}
&&\bar{\nabla}_{\bar{N}}u=\frac{1}{n}(\bar{\nabla}_{\bar{N}}V_{0}-\bar{\nabla}_{\bar{N}}V)=\frac{1}{n}\bar{g}(x,\bar{N})(1-e^{-2\omega}\bar{g}(x,E_{n}))-\frac{1}{n}e^{-\omega}\bar{g}(\bar{N},E_{n})\label{gradf}\\
&&=\frac{3}{2n}e^{\omega}(|x|^{2}+\frac{1}{3}x_{n})(1-x_{n})-\frac{3}{2n}(x_{n}+\frac{1}{3})\nonumber\\
&&=\frac{5|x|^{2}-1}{n(1-|x|^{2})}\nonumber.
\end{eqnarray}
Combining \eqref{f} with \eqref{gradf}, we get $$\bar{\nabla}_{\bar{N}}u=u\quad \hbox{on }T.$$

In summary, we see $u=\frac{1}{n}(V_{0}-V-1)$ is a solution of the partially overdetermined BVP \eqref{overd}, but $\Sigma$ is part of a horosphere.
\qed

\

\end{document}